\documentclass{amsart}

\usepackage[all]{xy}
\usepackage{amssymb}
\usepackage{amsmath}
\usepackage{amsfonts}
\usepackage{epsfig}
\usepackage{tikz}

\newtheorem{theorem}{Theorem}[section]
\newtheorem{prop}[theorem]{Proposition}
\newtheorem{lemma}[theorem]{Lemma}
\newtheorem{cor}[theorem]{Corollary}
\theoremstyle{definition}
\newtheorem{defi}[theorem]{Definition}
\newtheorem{rem}[theorem]{Remark}
\newtheorem{example}[theorem]{Example}

\newcommand{\sbt}{\,\begin{picture}(-1,1)(-1,-3)\circle*{2}\end{picture}\ }

\newcommand{\ihsk}{$\rm IHS-K3^{[2]}$ } 
\newcommand{\ihskpt}{$\rm IHS-K3^{[2]}$. } 

\newcommand{\lra}{\longrightarrow}

\newcommand{\ie}{i.e. }

\newcommand{\loccit}{loc. cit.}

\newcommand{\IP}{\mathbb{P}}
\newcommand{\IC}{\mathbb{C}}
\newcommand{\IR}{\mathbb{R}}
\newcommand{\IF}{\mathbb{F}}
\newcommand{\IZ}{\mathbb{Z}}
\newcommand{\IQ}{\mathbb{Q}}
\newcommand{\IN}{\mathbb{N}}

\newcommand{\cX}{\mathcal{X}}
\newcommand{\cO}{\mathcal{O}}
\newcommand{\frM}{\mathfrak{M}}

\newcommand{\coloneqq}{:=}

\DeclareMathOperator{\rank}{\rm{rank}}
\DeclareMathOperator{\trace}{tr}

\DeclareMathOperator{\Aut}{\rm Aut}
\DeclareMathOperator{\Def}{\rm Def}

\DeclareMathOperator{\Hom}{\rm Hom}
\DeclareMathOperator{\id}{\rm id}
\DeclareMathOperator{\NS}{\rm NS}

\DeclareMathOperator{\Pf}{\rm Pf}
\DeclareMathOperator{\Trans}{\rm Transc}
\DeclareMathOperator{\Sym}{\rm Sym}
\DeclareMathOperator{\Gr}{\rm Gr}
\DeclareMathOperator{\discr}{\rm discr}
\DeclareMathOperator{\Res}{\rm Res}

\begin{document}

\date{\today}

\title[Non-symplectic automorphisms]{Classification of automorphisms on a deformation family of hyperk\"ahler fourfolds\\ by $p$-elementary lattices}
\author{Samuel Boissi\`ere, Chiara Camere and Alessandra Sarti}

\address{Samuel Boissi\`ere, Universit\'e de Poitiers, 
Laboratoire de Math\'ematiques et Applications, 
 T\'el\'eport 2 
Boulevard Marie et Pierre Curie
 BP 30179,
86962 Futuroscope Chasseneuil Cedex, France}
\email{samuel.boissiere@math.univ-poitiers.fr}
\urladdr{http://www-math.sp2mi.univ-poitiers.fr/~sbossie/}

\address{Chiara Camere, Leibniz Universit\"at Hannover,
Institut f\"ur Algebraische Geometrie,
Welfengarten 1
30167 Hannover, Germany} 
\email{camere@math.uni-hannover.de}
\urladdr{http://www.iag.uni-hannover.de/~camere}

\address{Alessandra Sarti, Universit\'e de Poitiers, 
Laboratoire de Math\'ematiques et Applications, 
 T\'el\'eport 2 
Boulevard Marie et Pierre Curie
 BP 30179,
86962 Futuroscope Chasseneuil Cedex, France}
\email{sarti@math.univ-poitiers.fr}
\urladdr{http://www-math.sp2mi.univ-poitiers.fr/~sarti/}



\begin{abstract}
We give a classification of all non-symplectic automorphisms of prime order $p$ acting on irreducible 
holomorphic symplectic fourfolds deformation equivalent
to the Hilbert scheme of two points on a K3 surface, for $p=2,3$ and $7\leq p\leq 19$. Our classification 
relates the isometry classes of two natural lattices associated to the action of the automorphism on the second cohomology
group with integer coefficients with some invariants of the fixed locus and we
provide explicit examples. As an application, we find
new examples of \emph{non-natural} non-symplectic automorphisms. 
\end{abstract}

\maketitle

\section{Introduction}

Irreducible holomorphic symplectic (IHS) manifolds (or equivalently hyperk\"ahler manifolds), together with Calabi-Yau manifolds, are the natural higher dimensional 
generalizations of K3 surfaces. In particular many properties of automorphisms on K3 surfaces 
generalize to IHS manifolds (see \cite[\S 4]{Beauvillec1Nul}). The interest in automorphisms
of IHS manifold has  grown up extremely  in the last years (see \cite{BeauvilleInv, Camere, Mongardi, OS, BNWSenriques, BNWS}) especially the study of automorphisms of prime order on IHS
fourfolds deformation equivalent to the Hilbert scheme of two points on a K3 surface, that we call for short \ihskpt The case of 
symplectic automorphisms (i.e. those automorphisms leaving invariant the holomorphic two form) was studied by the second author \cite{Camere} for $p=2$ and  then completely settled by Mongardi~\cite{MongardiPhD} for all primes. They describe the fixed locus, which is never empty and  consists of isolated fixed points, abelian surfaces and K3 surfaces. The case of non-symplectic  involutions was considered first by Beauville \cite{BeauvilleInv} and recently by Ohashi--Wandel \cite{OW}. 
Here the authors study in detail families of \ihsk with 19 parameters and non-symplectic involution. In particular they describe
some {\it non-natural} involutions: these cannot be deformed to an involution on the Hilbert scheme of two points on a K3 surface induced by one on the K3 surface.

In this paper we classify the non-symplectic automorphisms of prime order $p\geq 3$ acting on \ihskpt As an application of our results, we construct the first known examples of non-natural non-symplectic automorphisms of order $3$ on \ihskpt This comes from the study of  non-symplectic automorphisms of order 3 on a special 20-dimensional and a special 14-dimensional family of Fano variety of lines on cubic fourfolds (Corollary~\ref{cor:nonnatural}). 

Let $X$ be an \ihskpt The study of non-symplectic automorphisms on $X$ leads to consider two natural lattices: the invariant sublattice $T$ of $H^2(X,\IZ)=U^{{\oplus}3}\oplus E_8^{{\oplus}2}\oplus \langle -2 \rangle$ and its orthogonal complement $S$. The lattice $T$ is contained in the N\'eron-Severi group of  $X$, while the lattice $S$ contains the transcendental lattice. These two lattices play an important role when studying moduli spaces. In the case of K3 surfaces, they also determine completely the topology of the fixed locus (see \cite{AS, AST}). In this paper, using lattice theory and a formula relating topological invariants of the fixed locus with lattice invariants (see \cite{BNWS}) we 
classify the lattices $S$ and $T$ when the order is $p=2,3$ and $7\leq p \leq 19$.
Our first main result (Theorem~\ref{th:lattice}) classifies all possible lattices $T$ and $S$ for $p\neq 2,5$. Our second main result (Theorem~\ref{ball}, Theorem~\ref{exi}) proves that all cases (except one) can be realized by an automorphism.
For $p=11, 13, 17, 19$ all the examples that we find are natural, for $p=3$ some examples are constructed using the Fano variety of lines of a cubic fourfold. In particular in a 12-dimensional family we find an example of non-symplectic automorphism of order 3 of {\it different kind}: it has  the same invariant lattice $T$ and orthogonal complement $S$ as a natural automorphism,  but its fixed locus is different (see Remark~\ref{rem:fixnotunique}). This is very surprising: it shows that in the case of \ihsk the lattice invariants do not uniquely determine  the fixed locus, in contrary to the case of K3 
surfaces. In several cases we construct coarse moduli spaces of  \ihsk with non-symplectic automorphism of order~$p$.
This construction uses the classification of non-symplectic automorphisms of order $p$  on
 K3 surfaces (Theorem~\ref{ball}). 

In the last section of the paper we discuss the case $p=2$. The situation is more complicated  because the lattice $T$ can have different embeddings in the lattice $U^{{\oplus}3}\oplus E_8^{{\oplus}2}\oplus \langle -2 \rangle$. This has an important influence in the construction of the moduli spaces. Our main result is Proposition~\ref{embedT} where we show that every embedding can be realized as the invariant lattice of a non-symplectic involution.
Although many cases can be realized using natural involutions, there are several cases where examples are still 
unknown.

\medskip

{\bf Aknowledgements:} We thank Alice Garbagnati, Klaus Hulek and Keiji Oguiso for 
helpful discussions and their interest in this work and Giovanni Mongardi, K\'evin Tari and Malte Wandel for useful remarks.

\section{Preliminary results on lattice theory}

A {\it lattice} $L$ is a free $\IZ$-module equipped with a nondegenerate symmetric bilinear form
$\langle \cdot, \cdot\rangle$ with integer values.  Its {\it dual lattice} is 
$L^{\vee}\coloneqq\Hom_{\IZ}(L,\IZ)$. It can be also described as follows:
$$
L^{\vee}\cong\{x\in L\otimes \IQ~|~\langle x,v\rangle\in \IZ\quad \forall v\in L\}.
$$
Clearly $L$ is a sublattice of $ L^{\vee}$ of the same rank, so the \emph{discriminant group} ${A_L:=L^{\vee}/L}$ is a finite abelian group whose order is denoted $\discr(L)$ and called the {\it discriminant of $L$}. We denote by
$\ell(A_L)$ the \emph{length} of $A_L$, \ie the minimal number of generators of $A_L$.  In a basis $\{e_i\}_i$ of~$L$, for the Gram matrix 
$M\coloneqq(\langle e_i,e_j\rangle)_{i,j}$ one has $\discr(\Lambda)=|\det(M)|$.

A lattice $L$ is called \emph{even} if $\langle x,x\rangle\in 2\IZ$ for all $x\in L$.  In this case   
the bilinear form induces a  quadratic form $q_L: A_L\lra \IQ/2\IZ$. Denoting by $(s_{(+)},s_{(-)})$ the signature of
$L\otimes\IR$, the triple of invariants $(s_{(+)},s_{(-)},q_L)$ characterizes the \emph{genus} of the even lattice $L$ (see \cite[Chapter 15, \S 7]{conwaysloane}, 
\cite[Corollary 1.9.4]{Nikulinintegral}).

A lattice $L$ is called {\it unimodular} if $A_L=\{0\}$. A sublattice $M\subset L$ is called \emph{primitive} if $L/M$ is a free $\IZ$-module.
If $L$ is unimodular and $M\subset L$ is a primitive sublattice, then $M$ and its orthogonal $M^\perp$ in $L$ have
isomorphic discriminant groups and $q_M=-q_{M^\perp}$.

Let $p$ be a prime number. A lattice $L$ is called $p$-\emph{elementary} if $A_L\cong\left(\frac{\IZ}{p\IZ}\right)^{\oplus a}$ for some
non negative integer $a$ (also called the \emph{length} $\ell(A_L)$ of $A$).  We write $\frac{\IZ}{p\IZ}(\alpha)$, $\alpha\in\IQ/2\IZ$ to denote that the quadratic form $q_L$ takes value $\alpha$ on the generator 
of the $\frac{\IZ}{p\IZ}$ component of the discriminant group. Recall the following classification result:

\begin{theorem}\label{Ruda}\cite[\S 1]{RS}.
\begin{enumerate}
\item An even, hyperbolic, $p$-elementary lattice of  rank $r$ with $p \neq 2$ and $r>2$ is uniquely determined by the integer $a$.
\item For $p\neq 2$, a hyperbolic $p$-elementary lattice with invariants $r,a$  exists if and only if
the following conditions are satisfied:
$a\leq r$, $r\equiv 0 \pmod 2$ and  
\begin{equation*}
\begin{cases}
\text{if } a\equiv 0 \pmod 2, & r\equiv 2 \pmod 4 \\
\text{if } a\equiv 1 \pmod 2, & p\equiv (-1)^{r/2-1} \pmod 4 .
\end{cases}
\end{equation*}
Moreover, if $r\not \equiv 2 \pmod 8$ then $r>a>0$.
\end{enumerate}
\end{theorem}

We formulate also the following generalization of Theorem~\ref{Ruda}, the proof is essentially contained in \cite[Chapter 15, \S 8.2]{conwaysloane},
we give it here again for convenience.
\begin{theorem}\label{indefigen}
Let $S$ be an even, indefinite, p-elementary lattice of rank $r\geq 3$, $p\geq 3$. Then S is uniquely determined
by its signature and its discriminant form.
\end{theorem}
\begin{proof}
 By a result of Eichler (see \cite[Chapter 15, Theorem 14]{conwaysloane}), since $r\geq 3$, the genus and the spinor genus of $S$ coincide, so by \cite[Chapter 15, Theorem 13]{conwaysloane} the genus contains only one isomorphism class. 
Now by \cite[Corollary 1.9.4]{Nikulinintegral} the genus of an even lattice is uniquely determined by the signature and the discriminant form.
\end{proof}

\begin{rem}\label{rem:2elem}
For $2$-elementary lattices the situation is different: an even indefinite $2$-elementary 
lattice is determined up to isometry by its signature, length and a third invariant $\delta\in\{0,1\}$. We refer to  
\cite[Theorem~1.5.2]{dolgabourbaki},  \cite[Theorem~4.3.1, Theorem~4.3.2]{Nikulinfactor}, \cite[\S 1]{RS} for the relations between these invariants.
\end{rem}

The following results on the unicity of the isometry class of a lattice of a given genus and on splitting of lattices will be needed in the sequel:
\begin{theorem}\label{kneser}\cite[Theorem 2.2]{morrison}.
Let $L$ be an even lattice of invariants $(s_{(+)}, s_{(-)}, q_L)$. Assume that $s_{(+)}>0$, $s_{(-)}>0$ and 
 $\ell(A_L)\leq \rank L-2$. Then up to isometry, $L$ is the only lattice with those invariants.   
\end{theorem}

\begin{theorem}\label{cs}\cite[Chapter 15, Theorem 21]{conwaysloane}. 
If $L$ is an indefinite lattice of rank $n$ and discriminant $d$ with more than one isometry class in its genus then
$4^{[\frac{n}{2}]} d$ is divisible by $k^{\binom{n}{2}}$ for some nonsquare natural number $k\equiv 0, 1 \mod 4$. 
\end{theorem} 

We  denote by $U$ the unique even unimodular hyperbolic 
lattice of rank two and by $A_k, D_h, E_l$ the even, negative definite lattices associated to the Dynkin diagrams 
of the corresponding type ($k\geq 1$, $h\geq 4$, $l=6,7,8$). We denote by $L(t)$ the lattice
whose bilinear form is the one on $L$ multiplied by $t\in\IN^\ast$. The following $p$-elementary lattices will be used in the sequel (see \cite{AST}):
\begin{itemize}
 \item For $p\equiv 3 \mod 4$, the lattice
$$
K_p :=\left(
\begin{array}{cc}
-(p+1)/2&1\\
1&-2\\
\end{array}
\right) 
$$
is negative definite and $p$-elementary with $a=1$. Note that $K_3= A_2$.

\item For $p\equiv 1 \mod 4$ the lattice 
$$
H_p:=\left(
\begin{array}{cc}
(p-1)/2&1\\
1&-2\\
\end{array}
\right) 
$$
is hyperbolic and $p$-elementary with $a=1$.

\item The lattice
$$
L_{17}:=\left(
\begin{array}{cccc}
-2&1&0&1\\
1&-2&0&0\\
0&0&-2&1\\
1&0&1&-4
\end{array}
\right) 
$$
is negative definite $17$-elementary with $a=1$. 

\item The lattice $E_6^\vee(3)$ is even, negative definite and $3$-elementary with $a=5$. To get a simple form of its discriminant group one can proceed as follows. By~\cite[Table 2]{AS} the lattice $U(3)\oplus E_6^\vee(3)$ admits a primitive
embedding in the unimodular K3 lattice with orthogonal complement isometric to $U\oplus U(3)\oplus A_2^{\oplus 5}$. It follows that the discriminant form of $E_6^\vee(3)$ is the opposite of those of $A_2^{\oplus 5}$, that is $\IZ/3\IZ(2/3)^{\oplus 5}$.
\end{itemize}

\begin{theorem}\cite[Theorem 1.13.5]{Nikulinintegral}\label{Niki}. Let $L$ be an even indefinite lattice 
of signature $(s_{(+)},s_{(-)})$ and assume that $s_{(+)}>0$ and $s_{(-)}>0$. Then:
\begin{enumerate}
\item If $s_{(+)}+s_{(-)}\geq 3+\ell(A_L)$ then $L\cong U\oplus W$ for a certain even lattice $W$.\\
\item If  $s_{(-)} \geq 8$ and $s_{(+)}+s_{(-)}\geq 9+\ell (A_L)$ then $L\cong E_8\oplus W'$ for a certain even lattice $W'$.
\end{enumerate}
\end{theorem}

The following result is an application of Nikulin's results on primitive embeddings~\cite{{Nikulinintegral}}.

\begin{prop}\label{discform}
Let $S$ be an even $p$-elementary lattice, $p\neq 2$,  with invariants $(s_{(+)},s_{(-)},q_S, a)$, and $L\coloneqq U^{\oplus 3}\oplus E_8^{\oplus 2}\oplus \langle -2\rangle$.
If $S$ admits a primitive embedding in~$L$, then the orthogonal complement $T$ of $S$ in $L$ has discriminant group $(\IZ/p\IZ)^{\oplus a}\oplus \IZ/2\IZ$
and discriminant form $(-q_S)\oplus q_L$. If moreover $s_{(+)}<3$, $s_{(-)}<20$ and $a\leq 21-\rank(S)$ then $T$ is uniquely determined
and there is at most one embedding of $S$ in $L$ up to an isomorphism of~$L$.
\end{prop}

\begin{proof}
The lattice $L$ has signature $(3,20)$ and discriminant form $q_L=\frac{\IZ}{2\IZ}\left(\frac{3}{2}\right)$ hence by Theorem~\ref{kneser} it 
is unique in its genus. By Nikulin~\cite[Proposition 1.15.1]{Nikulinintegral}, the data of a
primitive embedding of $S$ in $L$ is equivalent to the data of a quintuple $(H_S,H_L,\gamma,T,\gamma_T)$ satisfying
the following conditions:
\begin{itemize}
\item $H_S$ is a subgroup of $A_S=(\IZ/p\IZ)^{\oplus a}$, $H_L$ is a subgroup of $A_L=\IZ/2\IZ$ and $\gamma\colon H_S\to H_L$ is an isomorphism of groups such that for any $x\in H_S$, $q_L(\gamma(x))=q_S(x)$.
Here the only possibility is $H_S=\{0\}$, $H_L=\{0\}$ and $\gamma=\id$.
\item  $T$ is a lattice of invariants $(3-s_{(+)},20-s_{(-)},q_T)$ with $q_T=\left.\left((-q_S)\oplus q_L\right)\right|_{\Gamma^\perp/\Gamma}$,
where $\Gamma$ is the graph of $\gamma$ in $A_S\oplus A_L$, $\Gamma^{\perp}$ is the orthogonal complement of~$\Gamma$ in
$A_S\oplus A_L$ with respect to the bilinear form induced on $A_S\oplus A_L$, with values in $\IQ/\IZ$,
 and $\gamma_T$ is an automorphism of $A_T$ that preserves $q_T$. Moreover $T$ is the orthogonal complement of $S$
in $L$. Here we get $\Gamma=\{0\}$ hence $\Gamma^\perp=A_S\oplus A_L=A_T$ and $q_T=(-q_S)\oplus q_L$ is the only possibility. 
\end{itemize}
Since $p\neq 2$ one has $\ell(A_T)=a$. If $T$ is indefinite (that is $s_{(+)}<3$, ${s_{(-)}<20}$) and $a\leq\rank(T)-2=21-\rank(S)$ then by Theorem~\ref{kneser} the lattice $T$ is uniquely determined up to isometry.
Moreover, under these assumptions the natural homomorphism $O(T)\to O(A_T)$ is surjective (see \cite[Proposition~1.4.7]{dolgabourbaki}) so different choices of the isometry $\gamma_T$ produce isomorphic embeddings of~$S$ in~$L$(see \cite[Lemma~1.4.5]{dolgabourbaki}).
\end{proof}

%%%%%%%%%%%%%%%%%%%%%%%%%%%%%%%%%%%%%%%%%%%%%%%%%%%%%%%%%%%%%%%%%%%%%%%%%%%%%

\section{Automorphisms on deformations of the Hilbert scheme of two points on a K3 surface}

\subsection{Irreducible holomorphic symplectic manifolds}
A complex, compact, K\"ahler, smooth manifold $X$ is called \emph{irreducible holomorphic symplectic} (IHS) if~$X$ is simply connected and $H^0(X,\Omega_X^2)$ is spanned by an everywhere nondegenerate closed two-form, denoted by $\omega_X$. In dimension four, one of the most famous examples of IHS manifolds
are the Hilbert schemes $\Sigma^{[2]}$ of two points on a K3 surface~$\Sigma$.

The second cohomology space has a Hodge decomposition 
$$
H^2(X,\IC)=H^{2,0}(X)\oplus H^{1,1}(X)\oplus H^{0,2}(X)
$$
and we set $H^{1,1}(X)_\IR:=H^{1,1}(X)\cap H^2(X,\IR)$. The second cohomology group $H^2(X,\IZ)$ is torsion-free and  equipped with the Beauville--Bogomolov~\cite{Beauvillec1Nul} bilinear symmetric nondegenerate two-form of signature $(3,b_2(X)-3)$ with the property that, after scalar extension, $H^{1,1}(X)$ is orthogonal to $H^{2,0}(X)\oplus H^{0,2}(X)$. The \emph{N\'eron--Severi} group of $X$ is defined by:
$$
\NS(X):=H^{1,1}(X)_\IR\cap H^2(X,\IZ).
$$
We set $\rho(X):=\rank (\NS(X))$ the \emph{Picard number} of $X$ and $\Trans(X):=\NS(X)^\perp$ the orthogonal complement of $\NS(X)$ in $H^2(X,\IZ)$ for the quadratic form, called the \emph{transcendental lattice}. Note that $\NS(X)$ and $\Trans(X)$ are primitively embedded in $H^2(X,\IZ)$.
By \cite[Theorem 3.11]{Huybrechts} $X$ is projective if and only if $\NS(X)$ is a hyperbolic lattice.

Let $G\subset\Aut(X)$ be a finite group of automorphisms of prime order $p$ and fix a generator $\sigma \in G$. If $\sigma^*\omega_X=\omega_X$ then $G$ is called \emph{symplectic}. Otherwise, there exists a primitive $p$-th root of the unity $\xi$ such that $\sigma^*\omega_X=\xi\omega_X$ and $G$ is called \emph{non-symplectic}. Observe that non-symplectic actions exist only when $X$ is projective (see \cite[\S 4]{BeauvilleKaehler}).
Following the notation
of \cite{BNWS} we denote by $T:=T_G(X)$ the invariant sublattice of $H^2(X,\IZ)$ and by $S:=S_G(X)$ its orthogonal complement (see \cite[Lemma 6.1]{BNWS}). 

Since $h^0(X,TX)=0$ the variety $X$
admits a universal deformation ${p\colon\cX\to\Def(X)}$ where $p$ is a smooth proper holomorphic morphism,
$\Def(X)$ is a germ of  analytic space and $p^{-1}(0)\cong X$ (the isomorphism is part of the data). Although $h^2(X,TX)$ is not zero in general, $\Def(X)$ is smooth of dimension $h^1(X,TX)$ (see~\cite[\S 4]{HuybrechtsBourbaki}
and references therein). If $\Def(X)$
is taken small enough, all nearby fibres $\cX_t:=p^{-1}(t)$, $t\in\Def(X)$, are also IHS manifolds and the universal
deformation $p\colon\cX\to \Def(X)$ is in fact universal also for these fibres $\cX_t$ (see Huybrechts~\cite[\S~1.12]{Huybrechts}). Two IHS manifolds $X$ and $X'$ are called \emph{deformation equivalent} if there exists a smooth proper holomorphic morphism $p\colon\cX\to S$ with connected base $S$, two points $t,t'\in S$ and isomorphims $\cX_{t}\cong X$ and $\cX_{t'}\cong X'$.
We say that an IHS manifold $X$ is an \ihsk if it is deformation equivalent to the Hilbert scheme of two points on a K3 surface.

\subsection{Invariant sublattice and fixed locus}

We recall the main results of Boissi\`ere--Nieper-Wi{\ss}kirchen--Sarti~\cite{BNWS}.

\begin{prop}\cite[Definitions 4.5 and 4.9, Lemma 5.5]{BNWS}.\label{prop:SandT}
Let $X$ be an \ihsk and $G$ a 
finite group of prime order $p$  acting on $X$, with $3\leq p\leq 19$. Then:
\begin{itemize}
\item $\rank S=(p-1)m_G(X)$ for some positive integer $m_G(X)$. 
\item $\frac{H^2(X,\IZ)}{S\oplus T}\cong \left( \frac{\IZ}{p\IZ}\right)^{\oplus a_G(X)}$ for some non-negative integer
$a_G(X)$;
\item $A_T\cong \IZ/2\IZ\oplus (\IZ/p\IZ)^{\oplus a_G(X)}$, $A_S\cong (\IZ/p\IZ)^{\oplus a_G(X)}$;
\item If $G$ acts non-symplectically then $S$ has signature $(2,(p-1)m_G(X)-2)$ and $T$ has signature $(1, 22-(p-1)m_G(X))$. 
\end{itemize}
\end{prop}
  
We denote by $H^*(X^G,\IF_p)$ the cohomology of the fixed locus $X^G$ with coefficients in $\IF_p$ and we set $m:=m_G(X)$, $a:=a_G(X)$, $h^*(X^G, \IF_p)=\sum_{i\geq 0} h^i(X^G,\IF_p)$,
where $h^i(X^G,\IF_p):=\dim H^i(X^G,\IF_p)$. 

\begin{theorem}\cite[Theorem 6.15]{BNWS}\label{BNWS}
Let $X$ be an \ihsk and $G$ be a group of automorphisms of prime order $p$ on $X$ with $3\leq p\leq 19$, $p\not=5$. Then:
\begin{eqnarray}\label{fp}
\qquad h^*(X^G,\IF_p)=324-2a(25-a)-(p-2)m(25-2a)+\frac{1}{2}m((p-2)^2m-p)
\end{eqnarray}
with $2\leq (p-1)m<23$, $0\leq a\leq \min\{(p-1)m,~23-(p-1)m\}$ . 
\end{theorem}   

\begin{rem}\text{} 
\begin{itemize}
\item It follows from \cite[Proposition 5.17]{BNWS} that the fixed locus $X^G$ is never empty if $p\neq 2$. So one cannot
produce new examples of Enriques varieties (see \cite{BNWSenriques,OS}) using finite quotients of \ihsk other than quotients by (special) involutions.
\item If the group $G$ acts on the K3 surface $\Sigma$, it induces a natural action on $\Sigma^{[2]}$. 
One can similarly define the integers $a_G(\Sigma)$, $m_G(\Sigma)$ and it is easy to check that $a_G(\Sigma)=a_G(\Sigma^{[2]})$ and $m_G(\Sigma)=m_G(\Sigma^{[2]})$ (see \cite[Remark~5.16 (2)]{BNWS}).
For any $\sigma\in G$, considered as an automorphism of $\Sigma$, we denote by $\sigma^{[2]}$ the automorphism induced on $\Sigma^{[2]}$.
These automorphisms are called \emph{natural} in~\cite{sam} but we will use this term in a more general sense in Definition~\ref{nat}.
\end{itemize} 
\end{rem}
 
The topological Lefschetz fixed point formula gives a complementary information on the fixed locus $X^G$. Denote by $\chi(X^G):=\sum_i (-1)^i\dim H^i(X^G,\IR)$ the Euler characteristic of $X^G$. If $\sigma$ is a generator of $G$, one has:
$$
\chi(X^G)=\sum_{i\geq 0}(-1)^i\trace(\sigma^*_{|H^i(X,\IR)}).
$$
Since $X$ has real dimension $8$ and  trivial odd cohomology, using Poincar\'e duality the formula rewrites:
$$
\chi(X^G)=2+2\trace(\sigma^*_{|H^2(X,\IR)})+\trace(\sigma^*_{|H^4(X,\IR)}).
$$
Setting $r:=\rank T$ one sees easily that $\trace(\sigma^*_{|H^2(X,\IR)})=r-m$. 

\begin{lemma} One has
$\trace(\sigma^*|_{H^4(X,\IR)})= \frac{(m-r)(m-r-1)}{2}$.
\end{lemma}
\begin{proof} 
Denote by $\xi$ a primitive $p$-th root of the unity and by $V_{\xi^i}$ the one-dimensional representation
of $G$ with character $\xi^i$, for $i=1,\ldots,p-1$. Then, as a representation of $G$: 
$$
H^2(X,\IR)\cong \IR^{\oplus r}\oplus\bigoplus\limits_{i=1}^{p-1} V_{\xi^i}^{\oplus m}
$$
where $\IR=V_{\xi^0}$ stands for the trivial representation. Since $H^4(X,\IR)\cong\Sym^2H^2(X,\IR)$ (cf. \cite[Theorem~1.3]{Verbitsky}) one gets:
\begin{align*}
\Sym^2H^2(X,\IR)&\cong \IR^{\oplus \frac{r(r+1)}{2}}\oplus\bigoplus\limits_{i=1}^{p-1} V_{\xi^i}^{\oplus rm}
\oplus \bigoplus\limits_{i=1}^{p-1} \Sym^2\left(V_{\xi^i}^{\oplus m}\right) \\
&\quad\oplus\bigoplus\limits_{i=1}^{p-1}\bigoplus\limits_{j=i+1}^{p-1} \left(V_{\xi^i}^{\oplus m}\otimes V_{\xi^j}^{\oplus m}\right)
\end{align*}
Since $\Sym^2 \left(V_{\xi^i}^{\oplus m}\right)\cong V_{\xi^{2i}}^{\oplus\frac{m(m+1)}{2}}$
and $V_{\xi^i}^{\oplus m}\otimes V_{\xi^j}^{\oplus m}\cong V_{\xi^{i+j}}^{\oplus m^2}$ one gets
\begin{align*}
\trace\left(\sigma^\ast_{|H^4(X,\IR)}\right) &= \frac{r(r+1)}{2}+\left(\sum_{i=1}^{p-1} \xi^i\right) rm + \left(\sum_{i=1}^{p-1}\xi^{2i}\right) \frac{m(m+1)}{2}\\
&\quad +\left(\sum_{i=1}^{p-1}\sum_{j=i+1}^{p-1} \xi^{i+j}\right) m^2\\
&= \frac{r(r+1)}{2}-rm-\frac{m(m+1)}{2}+m^2= \frac{(m-r)(m-r-1)}{2}
\end{align*}
\end{proof}

Using the fact that $r=23-(p-1)m$ we obtain:
\begin{cor} The Euler characteristic of the fixed locus satisfies:
\begin{eqnarray}\label{chi}
\chi(X^G)=324-\frac{51}{2}mp+\frac{1}{2}m^2p^2.
\end{eqnarray}
\end{cor}

We deduce one further relation between the parameters $a$ and $m$:
\begin{cor}\label{inequality}
One has $a\leq m$.
\end{cor}
\begin{proof}
By the universal coefficient theorem we have that $H^i(X^G,\IZ)\otimes \IF_p$ injects in $H^i(X^G,\IF_p)$. Since $h^i(X^G,\IR)$ equals the rank of the free part of  $H^i(X^G,\IZ)$ it follows that $h^i(X^G,\IR)\leq h^i(X^G,\IF_p)$ for all $i$. So we get  $h^\ast(X^G,\IF_p)-\chi(X^G)\geq 0$.
Combining equations \eqref{fp} and \eqref{chi} we get:
$$
h^*(X^G,\IF_p)-\chi(X^G)=2(a-m)(a-25+mp-m).
$$
By Theorem \ref{BNWS} we have $a-25+mp-m< 0$, hence $a\leq m$.
\end{proof}

%%%%%%%%%%%%%%%%%%%%%%%%%%%%%%%%%%%%%%%%%%%%%%

\subsection{Computation of the invariant lattice}

Let $X$ be an \ihsk with a non-symplectic action of a group $G=\langle \sigma \rangle$  of prime order with $3\leq p\leq 19$, $p\not=5$. 
Recall that $H^2(X,\IZ)$ is isometric to the lattice $L=U^{\oplus 3}\oplus E_8^{\oplus 2}\oplus \langle -2\rangle$,  $T$ is the invariant sublattice of $H^2(X,\IZ)$ and $S$ is its orthogonal complement.

For each value of $p$, combining Proposition~\ref{prop:SandT}, the formulas \eqref{fp}, \eqref{chi} and Corollary~\ref{inequality}
we get all the possible values of $m:=m_G(X),a:=a_G(X)$ and we compute the values of $\chi:=\chi(X^G)$ and $h^*:=h^*(X^G,\IF_p)$. 
By Proposition~\ref{prop:SandT}, the lattice $S$ has signature $(2,(p-1)m-2)$ and is $p$-elementary 
with discriminant group $(\IZ/p\IZ)^{\oplus a}$ and the lattice $T$ has signature $(1,22-(p-1)m)$ 
and discriminant group $\IZ/2\IZ\oplus (\IZ/p\IZ)^{\oplus a}$. Considering $T$ and $S$ as sublattices of $L$ we call a triple $(p,m,a)$ \emph{admissible}
if such sublattices of $L$ with these invariants do exist. In this case, we compute their isometry class. We prove the following result:

\begin{theorem} \label{th:lattice} For every admissible value of $(p,m,a)$, there exists 
a unique even $p$-elementary lattice $S$ of signature $(2,(p-1)m-2)$ with $A_S=(\IZ/p\IZ)^{\oplus a}$. 
This lattice admits a primitive embedding in $L$, this embedding is unique if $(p,m,a)\notin\{(3,10,2),~(3,8,6),~(11,2,2)\}$, and its orthogonal complement $T$ in $L$ 
is uniquely determined by the signature $(1,22-(p-1)m)$ and the discriminant group 
$A_T=(\IZ/2\IZ)\oplus (\IZ/p\IZ)^{\oplus a}$. 
\end{theorem}

\begin{proof} The proof follows from a case-by-case analysis. We use Theorem~\ref{Niki} and Theorem~\ref{Ruda} on the existence of hyperbolic $p$-elementary lattices  to exclude the non-admissible values and determine the isometry class of $S$. The unicity of $T$ and of the embedding of $S$ in $L$ are a direct consequence of Proposition~\ref{discform}. Only a few special cases
require a more specific argument.
We first handle one case in detail to explain the method, then we treat the special cases.
\par The case $(p,m,a)=(3,9,1)$: here $S$ has signature $(2,16)$ and is $3$-elementary with $a=1$. Since $\rank(S)=18\geq 3+\ell(A_S)=4$ by Theorem~\ref{Niki} we can write $S=U\oplus W$ where $W$ is an even hyperbolic $3$-elementary lattice of signature $(1,15)$ and $a(W)=1$. We use Theorem~\ref{Ruda} to compute $W$ and to prove its unicity, this gives also the unicity of $S$, so one computes directly that the only possibility (up to isometry) is $W=U\oplus E_6\oplus E_8$. Finally $1=a\leq 21-\rank(S)$ so  by Proposition~\ref{discform}, $T$ is uniquely determined and the embedding of $S$ in $L$ is unique.

We discuss now the special cases.
 
\par The case $(p,m,a)=(3,11,1)$: one has $T=\langle 6\rangle$. It is easy to check that the homomorphism $O(T)\to O(A_T)$ is surjective so the argument of the proof of Proposition~\ref{discform} applies and $S$ admits a unique embedding in $L$.
\par The case $(p,m,a)=(3,1,1)$: in this case one computes directly that $S=A_2(-1)$ then one concludes with Proposition~\ref{discform}.
\par The case $(p,m,a)\in\{(3,10,2),(3,8,6),(11,2,2)\}$: Here we have $a>\rank(T)-2$.  Using Theorem~\ref{cs} one sees that $T$ is unique in its genus and one deduces its isometry class.

\par The case $(p,m,a)=(3,9,5)$: We prove that this case cannot occur. We compute as before that $S$ is isometric to $U^{\oplus 2}\oplus E_8\oplus E_6^\vee(3)$, so its discriminant group is $A_S=\frac{\IZ}{3\IZ}\left(\frac{2}{3}\right)^{\oplus 5}$. By Proposition~\ref{discform}, if $S$ admits a primitive embedding in $L$, then its orthogonal $T$ has signature $(1,4)$ and discriminant form $\frac{\IZ}{3\IZ}\left(\frac{4}{3}\right)^{\oplus 5}\oplus \frac{\IZ}{2\IZ}\left(\frac{3}{2}\right)$. Assume that such a lattice does exist. Consider its $3$-adic completion $T_3:=T\otimes_\IZ\IZ_3$. One has $A_{T_3}=(A_T)_3$ so $q_{T_3}=\frac{\IZ}{3\IZ}\left(\frac{4}{3}\right)^{\oplus 5}$. The rank of $T_3$ is then equal to $\ell(A_{T_3})$. By Nikulin~\cite[Theorem~1.9.1]{Nikulinintegral}, there exists a unique $3$-adic lattice $K$ of rank $\ell(A_{T_3})$ and discriminant form $q_{{T_3}}$. Then necessarily the determinants of the lattices $T_3$ and~$K$ (computed in some basis) 
differ by an invertible square:
$$
\det T=\det T_3\equiv \det K\mod (\IZ_3^\ast)^2.
$$
One has $\det T=(-1)^4|A_T|=2\cdot 3^5$. Using \cite[Proposition~1.8.1]{Nikulinintegral} one finds that $K=\langle 3\theta\rangle^{\oplus 5}$ with $\theta=\frac{1}{4}\in \IZ_3^\ast$ so $\det K=\left(\frac{3}{4}\right)^5$. The relation $\det T\equiv \det K\mod (\IZ_3^\ast)^2$ gives here $2^{11}\in(\IZ_3^\ast)^2$. This is not true (it would imply that $2$ is a square modulo $3$) so such a lattice $T$ does not exist.
\end{proof}

\begin{rem}
The isometry classes of the lattices $S$ and $T$ for all admissible values of $(p,m,a)$ are summarized
in Appendix~\ref{app:tables}
in the Tables \ref{ord3}, \ref{ord7}, \ref{ord11}, \ref{ord13}, \ref{ord17}, \ref{ord19} corresponding to  $p=3,7,11,13,17,19$.
The excluded values of $(p,m,a)$ are not written in the tables.
\end{rem}

\begin{prop}\label{unicityT}
Under the same assumptions as in Theorem~\ref{th:lattice}, the lattice $T$ admits a unique primitive embedding in the lattice $L$ whose orthogonal complement is $S$.
\end{prop}

\begin{proof}
The proof is essentially the same in Proposition \ref{discform}. Observe that in this case the orthogonal complement is given and it is isometric to $S$. So a primitive embedding of $T$ into $L$ corresponds to a  quadruple $(H_T,H_L,\gamma,\gamma_S)$. The only possibility is $H_T=H_L=\frac{\IZ}{2\IZ}$ so the only choice is $\gamma=\id$. Finally observe that all the lattices $S$ in the tables except the case $(p,m,a)=(3,1,1)$ have $\rank(S)\geq \ell(S)+2$ so by \cite[Proposition~1.4.7] {dolgabourbaki} the morphism $O(S)\to O(A_S)$ is surjective. In the case $(p,m,a)=(3,1,1)$ one shows the surjectivity by hand. Hence different choices of the isomorphism $\gamma_S$ produce isomorphic embeddings of $T$ in $L$. 
\end{proof}

%%%%%%%%%%%%%%%%%%%%%%%%%%%%%%%%%%%%%%%%%%%%%%

\section{Deformation of automorphisms on IHS manifolds}

Let $X$ be an IHS manifold and $f\in\Aut(X)$ be a biholomorphic automorphism of $X$. We denote by $p\colon \cX\to\Def(X)$, $p^{-1}(0)\cong X$, the universal deformation of $X$. By a theorem
of Horikawa~\cite[Theorem~8.1]{Horikawa3}, there exists an open neighbourhood $\Delta$ of the origin
of $\Def(X)$, a family of deformations $p'\colon \cX'\to\Delta$, $p'^{-1}(0)\cong X$, and a holomorphic map 
$\Phi\colon \cX_{|\Delta}\to \cX'$ over $\Delta$ such that $\Phi_0=f$. By the universality of~$p$, 
there exists a unique holomorphic map $\gamma\colon\Delta\to \Def(X)$ with $\gamma(0)=0$ such that $\cX'=\Delta\times_{\Def(X)}\cX$,
so we obtain by composition
a holomorphic map $F\colon\cX_{|\Delta}\to \cX$ such that $F_0=f$ with a commutative diagram
\begin{align}\label{diagr_deform}
\xymatrix{\cX_{\Delta}\ar[r]^F\ar[d]_{p_\Delta} & \cX\ar[d]^p\\\Delta\ar[r]^\gamma & \Def(X)}
\end{align}
Denote $D:=\Delta^\gamma=\{t\in\Delta\,|\,\gamma(t)=t\}$. By restricting to $D$ one obtains a  family of holomorphic maps
$$
\xymatrix{\cX_D\ar[r]^F\ar[d]_{p_D} & \cX_D\ar[d]^{p_D}\\D\ar@{=}[r] & D}
$$
with $F_0=f$ and such that the holomorphic map $F_t\colon \cX_t\to\cX_t$ is an isomorphism for all $t\in D$ (by shrinking $D$ if necessary). The pair $(p_D\colon\cX_D\to D, F)$ is thus a deformation space of the pair $(X,f)$ (see also \cite[Definition 1.1]{Mongardi}).
From the diagram (\ref{diagr_deform}) we get a commutative diagram of vector bundles over $X$ with exact rows
$$
\xymatrix{0\ar[r] & TX\ar[r]\ar[d]^{df} & \left.T{\cX_\Delta}\right|_X\ar[r]\ar[d]^{dF} & T_0\Delta\ar[r]\ar[d]^{d_0\gamma} & 0\\
0\ar[r] & TX\ar[r] & \left.T\cX\right|_X \ar[r] & T_0 S \ar[r]&0}
$$
that induces in cohomology an exact sequence
$$
\xymatrix{T_0\Delta\ar[r]^-{\rho}\ar[d]_{d_0\gamma} & H^1(X,TX)\ar[d]^{df}\\T_0 S\ar[r]^-\rho & H^1(X,TX)}
$$
where $\rho$ is the Kodaira-Spencer map. Since the deformation is universal, $\rho$ 
is an isomorphism (note that $T_0\Delta=T_0S$) and $df$ is an isomorphism since $f$ is an automorphism of $X$.
It follows that $d_0\gamma$ is invertible, so the fixed locus $D=\Delta^\gamma$ is smooth and its dimension
equals the dimension of the invariant space $H^1(X,TX)^{df}$.

If the automorphism $f$ acts symplectically on $X$ then
the isomorphism $TX\cong\Omega_X$ induced by $\omega_X$ is $f$-equivariant ($f$ induces natural actions denoted
 $df$ on tangent vectors, and denoted $f^\ast$ on differential forms, so $\dim D=\dim \left(H^{1,1}(X)^{f^\ast}\right)$. If the automorphism $f$ acts non-symplectically on $X$,
that is $f^\ast\omega_X=\xi\omega_X$ for some $\xi\in\IC^\ast$, $\xi\neq 1$, then the isomorphism $TX\cong \Omega_X$
is not $f$-equivariant and one computes that $H^1(X,TX)^{f^\ast}$ is isomorphic to the eigenspace of $H^{1,1}(X)$
corresponding to the eigenvalue $\xi$ of $f^\ast$. Assume that $f$ is an automorphism of prime order $p$, $G=\langle f\rangle$, so that
$\xi$ is a primitive $p$-th root of the unity. Since the action of $f^\ast$ on $H^2(X,\IC)$ comes from an action  on the lattice $H^2(X,\IZ)$  the eigenspaces of $f^\ast$ 
corresponding to the eigenvalues $\xi^i$, $i=1,\ldots,p-1$ have the same dimension which is $m_G(X)$. Since $H^{2,0}(X)$ is an eigenspace for the eigenvalue $\xi$ and $H^{0,2}(X)$ is one for $\bar\xi$, it follows that $\dim D=m_G(X)-1$ if $p\geq 3$ and $\dim D=m_G(X)-2$ if $p=2$.

If $X$ is an IHS manifold and $f\in\Aut(X)$, $G=\langle f\rangle$, 
it follows from Ehresmann's theorem that the
 $G$-module structure of $H^\ast(X,\IZ)$ is invariant under deformation of the pair $(X,f)$, so
the lattices $T_G(X)$, $S_G(X)$  and the values $h^\ast(X^G)$, $\chi(X^G)$  (by Formulas~(\ref{fp}),(\ref{chi})) are also invariant.

\begin{defi}\label{nat}
Let $X$ be an \ihskpt An automorphism $\sigma$
of $X$ is called {\it natural} if there exists a K3 surface $\Sigma$ and an automorphism $\varphi$ on $\Sigma$ such that
$(X,\sigma)$ is deformation equivalent to $(\Sigma^{[2]}, \varphi^{[2]})$, where $\varphi^{[2]}$ denotes the induced automorphism on $\Sigma^{[2]}$ by $\varphi$.     
\end{defi}

\section{Moduli spaces of lattice polarized IHS manifolds}

\subsection{The global Torelli theorem}

We recall some well known facts from \cite{Huybrechts} and \cite{Markmantorelli}. If $X$ is an irreducible holomorphic symplectic manifold, a {\it marking} for $X$ is a choice of 
an isometry $\eta: L \lra H^2(X,\IZ)$. Two marked pairs $(X_1,\eta_1)$ and $(X_2,\eta_2)$ are isomorphic if there is an isomorphism $f:X_1\lra X_2$ such that $\eta_1=f^*\circ \eta_2$. There exists a coarse moduli space $\mathfrak{M}_{L}$ that parametrizes isomorphism classes of marked pairs, which is a non-Hausdorff smooth complex manifold (see \cite{Huybrechts}). If $X$ is an \ihsk this has dimension $21$. Denote by 
$$
\Omega_L:=\{\omega\in\IP(L\otimes \IC)\, |\, q(\omega)=0, q(\omega+\bar{\omega})>0\}
$$
the {\it period domain} which is an open (in the usual topology) subset of the non-singular quadric defined by $q(\omega)=0$. The period map
$$
P:\mathfrak{M}_{L} \lra \Omega_L, (X,\eta)\mapsto \eta^{-1}(H^{2,0}(X))
$$ 
is a local isomorphism by the Local Torelli Theorem \cite[Th\'eor\`eme 5]{BeauvilleKaehler}. For $\omega\in \Omega_L$ we put
$$
L^{1,1}(\omega):=\{\lambda\in L\,|\, (\lambda,\omega)=0\}
$$
where $(\cdot,\cdot)$ is the bilinear form associated to the quadratic form $q$. 
Then $L^{1,1}(\omega)$ is a sublattice of $L$. Let $\lambda\in L$, $\lambda\not= 0$, and consider the hyperplane 
$$
H_{\lambda}=\{\omega\in\Omega_L\,|\, (\omega,\lambda)=0\}.
$$ 
Then  $L^{1,1}(\omega)=\{0\}$ if $\omega$ does not belong to the countable union of hyperplanes $\bigcup_{\lambda\in L\setminus\{0\}} H_{\lambda}$. In particular, given a marked pair $(X,\eta)$ we get  $\eta^{-1}(\NS(X))= L^{1,1}(P(X,\eta))$. The set
$\{\alpha\in H^{1,1}(X,\IR)\,|\, q(\alpha) >0\}$  has two connected components, we call {\it positive cone} $\mathfrak{C}_X$ the connected component containing the K\"ahler cone $\mathfrak{K}_X$.

Recall that two points $x,y$ of a topological space $M$ are called {\it inseparable} if every pairs of open neighbourhoods $x\in U$ and $y\in V$ have non-empty intersection, a point $x\in M$ is called a {\it Hausdorff point} if for every $y\in M$, $y\not= x$, then $x$ and $y$ are separable. 

\begin{theorem}[Global Torelli Theorem] \cite{Verbitsky},\cite[Theorem 2.2]{Markmantorelli} 

Let $\mathfrak{M}^0_L$  be a connected component of $\frM_L$.
\begin{enumerate}
\item The period map $P$ restricts to a surjective holomorphic map $$P_0:\frM^0_L\lra \Omega_L.$$
\item For each $\omega\in\Omega_L$, the fiber $P_0^{-1}(\omega)$ consists of pairwise inseparable points.   
\item Let $(X_1,\eta_1)$ and $(X_2,\eta_2)$ be two inseparable points of $\frM_L^0$. Then $X_1$ and~$X_2$ are bimeromorphic. 
\item The point $(X,\eta)\in\frM_L^0$ is Hausdorff  if and only if $\mathfrak{C}_X=\mathfrak{K}_X$.
\item\label{item5} The fiber $P_0^{-1}(\omega)$, $\omega\in \Omega_L$ consists of a single Hausdorff point if $L^{1,1}(\omega)$ is trivial, or if $L^{1,1}(\omega)$ is of rank one, generated by a class $\alpha$ satisfying $q(\alpha)>0$.   
\end{enumerate}
\end{theorem}

\begin{rem}
In assertion (\ref{item5}) if $L^{1,1}(\omega)$ is trivial then $X$ is non projective. If it is of rank one generated by a class $\alpha$ with $q(\alpha)>0$ then $X$ is projective, $\pm\alpha$ is an ample class and the  N\'eron--Severi group of $X$ has signature $(1,0)$, so its transcendental lattice has signature $(2, 20)$ (see \cite[Theorem~3.11]{Huybrechts}).
\end{rem}

%%%%%%%%%%%%%%%%%%%%%%%%%%%%%%%%%%%%%

\subsection{Lattice Polarizations}

In this subsection and the next one, we extend some constructions and results of \cite[\S 6]{DolgachevGeemenKondo}, \cite[\S 10]{dolgachevkondo}, \cite{dolgachevmirror} and \cite[\S 9]{AST}. 
See also \cite{Camere2} for related results.

Let $M$ be an even non-degenerate lattice of rank  $\rho\geq 1$ and signature $(1,\rho-1)$.
An \emph{$M$-polarized} \ihsk is a pair $(X,j)$  where $X$ is a projective \ihsk and $j$ is a primitive embedding of lattices $j:M\hookrightarrow NS(X)$. 
Two $M$-polarized \ihsk $(X_1, j_1)$ and $(X_2,j_2)$ are called \emph{equivalent} if there exists an isomorphism $f:X_1\to X_2$ such that $j_1=f^*\circ j_2$. As in \cite[\S 10]{dolgachevkondo} and \cite{dolgachevmirror} one can construct a moduli space of \emph{marked} $M$-polarized  \ihsk as follows. We fix a primitive embedding of $M$ in $L$ and we identify $M$ with its image in $L$. A  \emph{marking} of $(X,j)$ is an isomorphism of lattices $\eta: L \to H^2(X,\IZ)$ such that $\eta_{|M}=j$. As observed in \cite[p.11]{dolgachevmirror}, if the embedding of $M$ in $L$ is unique up to an isometry of $L$ then every $M$-polarization admits a compatible marking.
Two $M$-polarized marked \ihsk $(X_1,j_1, \eta_1)$ and $(X_2, j_2, \eta_2)$ are called \emph{equivalent} if there exists an isomorphism $f:X_1\to X_2$ such that $\eta_1=f^*\circ \eta_2$ (this clearly implies that $j_1=f^*\circ j_2$). Let $N:=M^{\perp}\cap L$ be the orthogonal complement of $M$ in $L$ and set
$$
\Omega_{M}:=\{\omega  \in \IP(N\otimes \IC)\,|\, q(\omega)=0,\, q(\omega+\bar{\omega})>0\}.
$$   
Since $N$ has signature $(2,21-\rho)$ the period domain $\Omega_{M}$ is a disjoint union of two connected components of  dimension $21-\rho$. For each $M$-polarized marked \ihsk $(X,j,\eta)$, 
since $\eta(M)\subset \NS(X)$ we have $\eta^{-1}(H^{2,0}(X))\in\Omega_M$. On the other hand, by the surjectivity of the period map \cite[Theorem 8.1]{Huybrechts}
restricted to some connected component $\mathfrak{M}^0_L$ of $\mathfrak{M}_L$ we can associate to each point $\omega\in\Omega_M$ an $M$-polarized \ihsk as follows:
there exists a marked pair $(X,\eta)\in\frM_L^0$ such that $\eta^{-1}(H^{2,0}(X))=\omega \in \IP(N\otimes \IC)$ so $M=N^{\perp}\subset \omega^{\perp}\cap L$, hence $\eta(M)\subset H^{2,0}(X)^{\perp}\cap H^{1,1}_{\IZ}(X)=\NS(X)$ and we take $(X,\eta_{|M},\eta)$.

By the Local Torelli Theorem for IHS manifolds, an $M$-polarized \ihsk  $(X,j)$ has a local deformation space $\Def_M(X)$ that is contractible, smooth of dimension $21-\rho$, such that the period map $P\colon\Def_M(X)\to\Omega_M$ is a local isomorphism (see \cite{dolgachevmirror}). By gluing all these local deformations spaces one obtains a moduli space $K_M$ of marked $M$-polarized \ihsk that is a non-separated 
analytic space, with a period 
map $P\colon K_M\to\Omega_M$.
The following proposition generalises \cite[Lemma 2.9]{OW}:

\begin{prop}\label{prop:dense}
If $\rank M\leq 20$, there exists a dense subset $\Omega^0_{M}$ of $\Omega_M$ such that if
 $(X,\eta)$ is a marked \ihsk whose period is in $\Omega_{M}^0$ then $\NS(X)$ is isomorphic to $M$.
\end{prop} 
\begin{proof}

For each $\lambda\in N\setminus\{0\}$ consider the hyperplane $H_{\lambda}:=\{\omega\in \Omega_M~|~(\omega,\lambda)=0\}$ and let $\mathcal{H}:=\bigcup_\lambda H_{\lambda}$. Each 
 subset $\Omega_M\setminus{H}_{\lambda}$ is open and dense in $\Omega_M$ hence  by Baire's Theorem 
the subset $\Omega^0_{M}:=\Omega_{M}\setminus\mathcal{H}$ is dense in $\Omega_M$ since~$\mathcal{H}$ is a countable union of 
complex closed subspaces. If $\omega=P(X,\eta)\in\Omega^0_M$ then $\NS(X)=\eta(L^{1,1}(\omega))=\eta(M)$. 
\end{proof}

This means hence that for a \emph{general} point of $\Omega_M$ the associated marked $M$-polarized 
\ihsk has N\'eron--Severi group isometric to $M$ and transcendental lattice isometric to $N$. If $\rank M=21$
then $\Omega_M$ consists of two periods that correspond to an \ihsk whose N\'eron--Severi group is isometric
to $M$.
We specialize this construction of the period domain in the case of projective \ihsk with non-symplectic automorphism.
\begin{rem}
Observe that in the construction we fix an embedding of $M$ in $L$. Different embeddings give a priori different constructions of $\Omega_M$.
\end{rem}

%%%%%%%%%%%%%%%%%%%%%%%%%%%%%%%%%%%

\subsection{Eigenperiods of projective \ihsk with a non-symplectic automorphism}

Let $(X,j)$ be an $M$-polarized \ihsk and $G=\langle\sigma\rangle$ a cyclic group of prime 
order $p\geq 2$ acting non-symplectically on $X$. It is easy to see that the invariant sublattice $T=T_G(X)$ is contained in $\NS(X)$. Assume that the action of $G$ on $j(M)$ is the identity
and that there exists a group homomorphism $\rho: G\lra O(L)$ such that
$$
M=L^{\rho}:=\{x\in L\,|\, \rho(g)(x)=x,\forall g\in G\}.
$$
We define a {\it $(\rho, M)$-polarization} of $(X,j)$ as a marking $\eta\colon L\to H^2(X,\IZ)$ such that $\eta_{|M}=j$ and $\sigma^\ast=\eta\circ\rho(\sigma)\circ\eta^{-1}$.

Two $(\rho, M)$-polarized \ihsk $(X_1, j_1)$ and $(X_2,j_2)$ are isomorphic if there are markings, $\eta_1\colon L\to H^2(X_1, \IZ)$ and $\eta_2\colon L\lra H^2(X_2,\IZ)$ such that $\eta_i{_{|M}}=j_i$, and an isomorphism $f\colon X_1\to X_2$ such that
$\eta_1=f^*\circ \eta_2$. 

Recall that by construction $\omega_X$ is the line in $L\otimes \IC$ defined by $\omega_X=\eta^{-1}(H^{2,0}(X))$. Let  $\xi\in\IC^*$ such that $\rho(\sigma)(\omega_X)=\xi \omega_X$. Observe that $\xi\not= 1$ since the action is non-symplectic and it is a primitive $p$-th root of the unity since $p$ is prime. The period $\omega_X$ belongs to the eigenspace of $N\otimes \IC$
 relative to the eigenvalue $\xi$, where $N=M^{\perp}\cap L$. We denote it by $N(\xi)$ (if $p=2$, we have $\xi=-1$ and  we denote $N(\xi)=N_{\IR}(\xi)\otimes \IC$, where $N_{\IR}(\xi)$ is the real eigenspace relative to $\xi=-1$). 
 
Assume that $\xi\not= -1$, then the period belongs to the space
$$
\Omega_M^{\rho,\xi}:=\{x\in \IP(N(\xi))\,|\, q(x+\bar{x})>0\}
$$ 
of dimension
 $\dim N(\xi)-1$ which is a complex ball if $\dim N(\xi)\geq 2$. It is easy to check that every point $x\in \Omega_M^{\rho,\xi}$  satisfies automatically the condition $q(x)=0$. 

If $\xi=-1$  we set $\Omega_M^{\rho,\xi}:=\{x\in \IP(N(\xi))\,|\, q(x)=0, \, q(x+\bar{x})>0\}$.
It has dimension $\dim N(\xi)-2$, clearly  $\Omega_M^{\rho,\xi}\subset \Omega_M$. 

Assume now that $M=T=\overline{T}\oplus \langle -2\rangle$ where $\overline{T}$ is an even non degenerate lattice of signature $(1,21-(p-1) m)$. Assume moreover that $\overline{T}$ has a primitive embedding in the K3 lattice
 $\Lambda$, we fix such an embedding and we call again $\overline{T}$ the image. This induces in a natural way a primitive embedding of $T$ in $L$. We identify then $T$ with its image.  Let $N=S=T^{\perp}\cap L$ and assume that $S\subset \Lambda$. For $\delta\in S$ with $q(\delta)=-2$, denote $H_\delta=\delta^{\perp}\cap S$ and $\Delta:=\bigcup_{\delta\in S, q(\delta)=-2} H_{\delta}$. Then we have (with the same notation as above):
 
\begin{theorem}\label{ball}
Let $X$ be a $(\rho, T)$-polarized \ihsk such that $H^{2,0}(X)$ is contained in the eigenspace of $H^2(X,\IC)$ relative to $\xi$. 
 Then $\omega_X\in \Omega_T^{\rho,\xi}$ and
conversely, if  $\dim N(\xi)\geq 2$ every point of $\Omega_T^{\rho,\xi}\setminus\Delta$ is the period point of some $(\rho, T)$-polarized \ihskpt
\end{theorem}

The proof is an application of a result of Namikawa that we recall for convenience.

\begin{theorem}\cite[Theorem 3.10]{Namikawa}\label{namik}
Let $\Sigma$ be a K3 surface and $G$ a finite subgroup of the group of isometries of $H^2(\Sigma,\IZ)$. Denote by $\omega$ the period of $\Sigma$ in $H^2(\Sigma,\IC)$ and set $S_{G}(\Sigma):=(H^2(\Sigma,\IZ)^G)^{\perp}\cap\{\IC\omega\}^{\perp}$. Then there exists an element $\varpi$ in the Weyl group $W(\Sigma)$ of $\Sigma$ such that $\varpi G\varpi^{-1}\subset \Aut(\Sigma)$ if and only if 
\begin{itemize}
\item[(i)] $\IC\omega$ is $G$-invariant.
\item[(ii)] $S_{G}(\Sigma)$ contains no element of length $-2$.
\item[(iii)] If $\omega\in H^2(\Sigma,\IC)^G$ then $S_{G}(\Sigma)$ is non-degenerate and negative definite if $S_{G}(\Sigma)\not= 0$.
\item[(iii')] If $\omega\notin H^2(\Sigma,\IC)^G$ then $H^2(\Sigma,\IC)^G$ contains an element $\alpha$ with $(\alpha,\alpha)>0$.
\end{itemize} 
\end{theorem}

\begin{proof}[Proof of Theorem \ref{ball}] We have already proven that $\omega_X\in \Omega_T^{\rho,\xi}$. 
Conversely, if $\dim N(\xi)\geq 2$ the locus $\Omega_T^{\rho,\xi}\setminus\Delta$ is non empty (see the proof of 
Proposition~\ref{prop:dense}) so for any $\omega\in\Omega_T^{\rho,\xi}\setminus\Delta$, by our assumptions on the lattice $T$ the map $\rho(\sigma)$ acts as the identity on the $\langle-2\rangle$ class of the decomposition $\Lambda\oplus \langle-2\rangle=L$, so it leaves invariant the lattice $\Lambda$ and it thus induces a corestriction map to $\Lambda$:
$$
\bar{\rho}\colon G\lra O(\Lambda)
$$
with the property that $\overline{T}=\Lambda^{\bar\rho(\sigma)}$ and by our assumption on $S$ the orthogonal complement of $\overline{T}$ in $\Lambda$ is $S$.
Now as in \cite[\S 11]{dolgachevkondo} we consider the period domain of $(\bar \rho, \bar T)$-polarized K3 surfaces. For $\xi\not= -1$ this is
$$
D_{\overline{T}}^{\rho,\xi}:=\{x\in \IP(S(\xi))\,|\, (x,\bar{x})>0\}
$$ 
and for $\xi=-1$ this is 
$$
D_{\overline{T}}^{\rho,\xi}:=\{x\in \IP(S(\xi))\,|\, (x,x)=0, \,(x,\bar{x})>0\}.
$$ 
Observe that $D_{\overline{T}}^{\rho,\xi}\setminus \Delta$ is not empty because it coincides with  $\Omega_T^{\rho,\xi}\setminus\Delta$ (which is not empty by assumption). The point $\omega\in\Omega_T^{\rho,\xi}\setminus\Delta$ is  then also a point in $D_{\overline{T}}^{\rho,\xi}\setminus\Delta$, hence there exists a $\overline{T}$-polarized K3 surface
$\Sigma$ (see \cite[\S 9]{AST} and \cite[\S 11]{dolgachevkondo}). It has a 
$\overline{T}$-polarization  $\eta: \Lambda\lra H^2(\Sigma,\IZ)$ with $\eta_{|\overline{T}}(\overline{T})\subset \NS(\Sigma)$ and  $\eta(\omega)=H^{2,0}(\Sigma)$. The isometry  $\eta\circ\bar \rho(\sigma)\circ \eta^{-1}$ acts on $H^2(\Sigma,\IZ)$, it is the identity on $j(\overline{T})$ (where $j:=\eta_{|\overline{T}}$) and it preserves $H^{2,0}(\Sigma)$. Let us check that the conditions of Theorem~\ref{namik} are satisfied. Since the line generated by $\omega$ is preserved by $\bar\rho(\sigma)$ we have condition~$i)$. By assumption $S\cap\{\omega\}^{\perp}\cap \Lambda$ does not contain classes of length $-2$, this gives condition~$ii)$. By construction $\overline{T}$ is the $\bar\rho(\sigma)$-invariant sublattice of $\Lambda$  and it is hyperbolic, so it contains a class $\alpha$ with $(\alpha,\alpha)>0$, this gives condition  $iii)'$. 
So there exits $a\in\Aut(\Sigma)$ with  $a^*=\eta\circ\bar \rho(\sigma)\circ \eta^{-1}$ (up to conjugation with an element of the Weyl group). Take  $Y:=\Sigma^{[2]}$ with the marking $L\to H^2(Y,\IZ)=H^2(\Sigma, \IZ)\oplus \langle-2\rangle$ in such a way that its restriction to $\Lambda$ is equal to $\eta$. We still denote this marking by $\eta$. By construction $Y$ admits an automorphism $a^{[2]}$ such that 
$(a^{[2]})^*=\eta\circ\rho(\sigma)\circ \eta^{-1}$, so $Y$ is  $(\rho,T)$-polarized with period $\omega_Y=\omega$. 
\end{proof}

\begin{cor}\label{unicity}
  Let $X$, resp. $X'$ be two \ihsk admitting non-symplectic automorphisms $\sigma$, resp. $\sigma'$ 
of the same order. Assume as above that  $T_{\sigma}(X)=\overline{T}\oplus \langle-2\rangle=T_{\sigma'}(X)$ and $S_{\sigma}(X)=S_{\sigma'}(X)$ have rank at least $p$. 
Then there exists $h\in O(L)$ such that $h\circ \rho(\sigma) \circ h^{-1}=\rho'(\sigma')$
(that is: once fixed the order, the action of $\sigma$ on $L$ is uniquely determined by $T_{\sigma}$ 
and $S_{\sigma}$).
\end{cor}

\begin{proof}
We use the same notation as above. As in the proof of Theorem \ref{ball} we can associate to $X$ and $X'$ two K3 surfaces  $\Sigma$  and $\Sigma'$ with 
respective automorphisms $\bar \sigma$ and $\bar \sigma'$ such that $\overline{T}_{\bar \sigma}(\Sigma)= \overline{T}_{\bar \sigma'}(\Sigma')$
 and $S_{\bar \sigma}(\Sigma)=S_{\bar \sigma'}(\Sigma')$. Then by \cite[Proposition 9.3]{AST} $\bar \sigma^*$ and $\bar \sigma'^*$ are conjugated via an element of $ O(\Lambda)$. More precisely if $\bar\rho\colon\langle\sigma\rangle\lra O(\Lambda)$ and $\bar\rho'\colon\langle\sigma'\rangle\lra O(\Lambda)$ are the corestrictions of the analogous morphisms with image in $O(L)$, then
there exists $\bar{h}\in O(\Lambda)$ such that $\bar{h}\circ \bar\rho(\sigma) \circ \bar{h}^{-1}=\bar\rho'(\sigma')$. We define an isometry  $h$ 
of $L=\Lambda\oplus \langle-2\rangle$ by: $h_{|\Lambda}=\bar{h}$  and $h_{|\langle -2\rangle}=\id$. Recall that:
$$
\rho(\sigma)_{|\langle -2\rangle}=\id, \qquad \rho'(\sigma')_{|\langle -2\rangle}=\id.
$$
This implies that $h\circ \rho(\sigma) \circ h^{-1}=\rho'(\sigma')$, hence the statement.
\end{proof}

%%%%%%%%%%%%%%%%%%%%%%%%%%%%%%%%%%%%%%%%%%%%%%%%%%%%%%%%%%%%%%

\section{Examples of automorphisms} \label{s:exampl}

\subsection{Natural automorphisms on Hilbert schemes of points}\label{natural}
Let $\Sigma$ be a K3 surface and $\varphi$ be a non-symplectic automorphism of prime order $p\geq 3$ acting on~$\Sigma$.
By \cite[Theorem 0.1]{AST} the fixed locus $\Sigma^\varphi$ is a disjoint union of curves and points of the form
$$
\Sigma^{\varphi}=C_g\cup R_1\cup\ldots\cup R_k\cup \{p_1,\ldots,p_N\}
$$
where $C_g$ is a smooth curve of genus $g\geq 0$, $R_i$, $i=1,\ldots,k$, are smooth rational curves, and $p_j$,
$j=1,\ldots,N$ are isolated fixed points.
%   Moreover $N\leq 9$, $k\leq 5$.
The isolated fixed points are of different types depending on the local action of $\varphi$ at them. The possible local actions are
\[A_{p,t}=\left(
\begin{array}{cc}
\xi_p^{t+1}&0\\
0&\xi_p^{p-t}\\
\end{array}
\right),\ \ t=0,\dots,p-2; \] here we use the
same notation as in \cite{AST} and we denote by $n_t$ the number of isolated fixed
points corresponding to the local action $A_{p,t}$.

 By using the results of \cite[\S 4.2]{sam} one computes the fixed locus of $\varphi^{[2]}$ on $\Sigma^{[2]}$. It consists of:
\begin{itemize}
\item $N(N-1)/2+2(N-n_{\frac{p-1}{2}})$ isolated fixed points. For this contribution, one has
$\chi=h^*=N(N-1)/2+2(N-n_{\frac{p-1}{2}})$.
\item $(n_{\frac{p-1}{2}}+Nk+k)$ smooth rational curves: $n_{\frac{p-1}{2}}$ for each fixed point at which the local action has two equal
eigenvalues, $Nk$ for each couple of a fixed point and a rational curve, and $k$ for each rational curve. This last case comes from the fact that
taking the schemes of length 2 of $\Sigma^{[2]}$ over a point of a rational curve, we get a curve isomorphic to $\IP^1$ contained in the exceptional
set (we also get a surface isomorphic to $\IP^2=(\IP^1)^{[2]}$, this contribution is taken into account below). For this contribution, one has
$\chi=h^*=2(n_{\frac{p-1}{2}}+Nk+k)$.
\item $N+1$ curves isomorphic to $C_g$, one for each couple consisting of a fixed point $p_j$ and the curve $C_g$, and one for the curve $C_g$.
The explanation for this last contribution is the same as above in the case of the rational curves. Here $\chi=(N+1)(2-2g)$,
$h^*=(N+1)(2+2g)$. 
\item $k(k-1)/2$ surfaces isomorphic to $\IP^1\times\IP^1$, one for each couple of distinct rational curves.
Here $\chi=h^*=4 (k(k-1)/2)$.
\item $k$ surfaces isomorphic to $\IP^1\times C_g$, one for each rational curve. Here $\chi=(4-4g) k$, $h^*=(4+4g)k$.
\item $k$ surfaces isomorphic to $\IP^2=(\IP^1)^{[2]}$, one for each rational curve. Here $\chi=h^*=3k$.
\item one surface $(C_g)^{[2]}$, which is the Hilbert scheme of two points on $C_g$. Here $\chi=3+2g^2-5g$, $h^*=3+2g^2+3g$.   
\end{itemize} 
As a consequence, the fixed locus of $\varphi^{[2]}$ on $\Sigma^{[2]}$ has invariants:
\begin{align*}
\chi&=2(n_{\frac{p-1}{2}}+Nk+k)+N(N-1)/2\\
&\quad+2(N-n_{\frac{p-1}{2}})+4k(k-1)/2+3k+(N+1)(2-2g)\\
&\quad+k(4-4g)+3-5g+2g^2\\
&=(1/2)(2g-2-N-2k)(2g-5-N-2k),\\
h^*&=N^2/2+7N/2+2Nk+7k+2Ng+5+5g+2k^2+4kg+2g^2.\\
\end{align*}
Observe that $\varphi^{[2]}$ acts non-symplectically on $\Sigma^{[2]}$. In fact there is an injective
morphism $\iota: H^2(\Sigma,\IC)\lra H^2(\Sigma^{[2]}, \IC)$ such that 
$$
H^2(\Sigma^{[2]}, \IC)=\iota(H^2(\Sigma,\IC))\oplus \IC[E]
$$ 
where $E$ denotes the exceptional set and $\iota$ respects the Hodge decomposition  (see \cite[Proposition 6]{BeauvilleKaehler}). 
If $\alpha\in H^{2,0}(\Sigma)$ then $\iota(\alpha)\in H^{2,0}(\Sigma^{[2]})$ and by the definition of~$\iota$ 
one has $\varphi^{[2]}(\iota(\alpha))=\iota(\varphi(\alpha))$ (see \cite{BSbirat}).  
This implies that $S_\varphi(\Sigma)=S_{\varphi^{[2]}}(\Sigma^{[2]})$ and $\langle -2 \rangle \oplus T_\varphi(\Sigma)=T_{\varphi^{[2]}}(\Sigma^{[2]})$.
In the tables \ref{ord3}, \ref{ord7}, \ref{ord11}, \ref{ord13}, \ref{ord17}, \ref{ord19} we mark with a  $\clubsuit$ the cases that are realised with natural automorphisms.

\begin{rem}
Theorem \ref{BNWS} for order $5$ automorphisms  holds only for natural automorphisms, 
so in Table~\ref{ord5} the list of admissible triples $(5,m,a)$ is only in this special case. 
For the moment, there are no other known examples of non-symplectic
automorphisms of order five (see also Remark~\ref{remvictor}).
\end{rem}

%%%%%%%%%%%%%%%%%%%%%%%%%%%%%%%%%%%%%%%%%%%%%%%%%%%%%%%%%%%%%%%%%%%%%%%%%%%%

\subsection{The Fano variety of lines on a cubic fourfold}\label{fano}
Let $V$ be a smooth cubic hypersurface in $\IP^5$, the  {\it Fano variety of lines} on $V$ is defined as 
$$
F(V):=\{l\subset \Gr(1,5)|l\subset V\}.
$$
By \cite{beauvilledonagi} the variety $F(V)$ is an \ihskpt 
One can construct examples of non-symplectic automorphisms of prime order by starting with an automorphism
of a cubic hypersurface in $\IP^5$ and then looking at the induced automorphism on $F(V)$.
By a classical result of Matsumura and Monsky \cite{mm} the automorphism group of a cubic hypersuface in $\IP^5$ is finite, the automorphisms are induced by
linear automorphisms of $\IP^5$ and a generic cubic hypersurface in $\IP^5$ has no automorphism. All automorphisms of prime 
order of smooth cubic fourfolds are classified in \cite[Theorem 3.8]{victor}. We are interested in those that induce a non-symplectic automorphism on $F(V)$. 
 
Denote by $Z\subset F(V)\times V$ the universal family and by 
$p$, resp. $q$, the projection to $F(V)$, resp. $V$.
By \cite[Proposition 4]{beauvilledonagi} the 
Abel-Jacobi map 
$$
A:=p_\ast q^\ast\colon H^4(V,\IZ)\lra H^2(F(V),\IZ)
$$ 
is an isomorphism of Hodge 
structures with $A(H^{3,1}(V))\cong H^{2,0}(F(V))$. If $\sigma$ is an automorphism of $V$, by equivariance
of $A$ one has $\sigma^\ast(\omega_{F(V)})=A(\sigma^\ast(A^{-1}(\omega_{F(V)})))$.
Let $f:=f(x_0,\ldots,x_5)$ be the cubic polynomial whose zero set in $\IP^5$ is $V$. 
By Griffith's residue theorem (see~\cite[Proposition~18.2]{voisin}) the cohomology group $H^{3,1}(V)$ is one-dimensional
and generated by the residue $\Res\left(\frac{\omega_{\IP^5}}{f^2}\right)$, where 
$$
\omega_{\IP^5}=\sum_{i} (-1)^i x_i dx_0\wedge\ldots\wedge \widehat{dx_i}\wedge \ldots \wedge dx_5.
$$
Assume that $\sigma$ is an automorphism of $\IP^5$ leaving $V$ globally invariant. The form $\frac{\omega_{\IP^5}}{f^2}$ is a closed meromorphic $5$-form on $\IP^5$, holomorphic on $\IP^5\setminus V$, with poles of order two along $V$, that 
we consider as a differential form on $U:=\IP^5\setminus V$. 
Taking its cohomology class we get an element in $H^5(U,\IC)$. There is a natural 
$\sigma$-equivariant isomorphism $H^5(U,\IC)\cong H^5(\IP^5,\Omega^\bullet_{\IP^5}(\log V))$. 
For a form with logarithmic pole $\alpha\wedge \frac{df}{f}\in \Omega^\bullet_{\IP^5}(\log V)$,
with $\alpha\in\Omega_{\IP^5}^{\bullet-1}$, 
one has by definition $\Res\left(\alpha\wedge \frac{df}{f}\right)=2\rm{i}\pi\alpha_{|V}$ (see \cite[\S18.1.1]{voisin}). 
Assume that $\sigma$ acts diagonally, that is $\sigma(x_0,\ldots,x_5)=(\alpha_0x_0,\ldots,\alpha_5x_5)$ and denote $\det(\sigma):=\prod_i\alpha_i$. Then $\sigma^\ast\omega_{\IP^5}=\det(\sigma)\omega_{\IP^5}$. 
Assume furthermore that $f$ is a projective invariant for $\sigma$, that is $\sigma^\ast f=\lambda_\sigma f$ with $\lambda_\sigma\in\IC^\ast$. Then $\sigma^\ast \left(\frac{df}{f}\right)=\frac{df}{f}$ so the map $\Res\colon H^5(U,\IC)\to H^4(V,\IC)$ is 
$\sigma$-equivariant. It follows that:
$$
\sigma^\ast\Res\left(\frac{\omega_{\IP^5}}{f^2}\right)=\frac{\det(\sigma)}{\lambda_\sigma^2}\Res\left(\frac{\omega_{\IP^5}}{f^2}\right).
$$
This proves the following result:

\begin{lemma}\label{actioncubic}
Let $\sigma$ be a diagonal automorphism of $V$. 
If the homogeneous polynomial of degree three $f\in \IC[x_0,\ldots, x_5]$ which defines $V$ is a
projective invariant for the action of $\sigma$, with $\sigma^\ast f=\lambda_\sigma f$ then the action of $\sigma$ on $F(V)$ is non-symplectic if and only if $\frac{\det(\sigma)}{\lambda_\sigma^2}\not=1$.   
\end{lemma}

Looking inside the classification of \cite[Theorem 3.8]{victor} we see that examples of
non-symplectic automorphisms occur only for $p=2,3$. In this section we consider only
the case $p=3$ and we find four families of examples, for which we compute now the fixed locus and the lattices
 $T$ and $S$. Put $\xi:=\exp(2\pi i/3)$.

\begin{rem}\label{remvictor}
There is a small mistake in the classification of \cite{victor}: the cubics 
of the family denoted $\mathcal{F}_5^2$, that would have a non-symplectic automorphism 
of order $5$, are in fact all singular (as confirmed by the authors). So this family should not
be in the list.
\end{rem}

\begin{example}\label{fanononnat1}(Case $T=\langle 6\rangle$). Consider the automorphism of order $3$ of $\IP^5$ given by:
$$
\sigma_1(x_0:x_1:x_2:x_3:x_4:x_5)=(x_0:x_1:x_2:x_3:x_4:\xi x_5).
$$
The family of invariant cubics is 
$$
V_1\colon L_3(x_0,\ldots,x_4)+x_5^3=0
$$
where $L_3$ is a homogeneous polynomial of degree $3$. The fixed locus of $\sigma_1$ on $V$ is the
cubic 3-fold $\mathcal{C}:=\{x_5=0,\, L_3(x_0,\ldots,x_4)=0\}$. The fixed points on $F(V_1)$ correspond to 
$\sigma_1$-invariant lines on $V$. If a line $L\subset V$ is invariant then either it is
pointwise fixed or it contains two fixed points. If $L$ is pointwise fixed then it is contained
in $\mathcal{C}$, if it is only invariant but not pointwise fixed,  then $L$ intersects  $\mathcal{C}$ in two points.  Hence  it intersects also the hyperplane $\{x_5=0\}$ in two points so it is contained in it.
Hence $L\subset V\cap  \{x_5=0\}$, this means that $L$ is contained in $\mathcal{C}$. We call again $\sigma_1$ the induced automorphism on $F(V_1)$, this is non-symplectic by Lemma~\ref{actioncubic}.
We have shown that its fixed locus is the Fano surface $F(\mathcal{C})$ of $\mathcal{C}$. This is a well known surface of general type with Hodge numbers: $h^{1,0}=h^{0,1}=5$, $h^{0,2}=h^{2,0}=10$, $h^{1,1}=25$. Hence one computes $\chi(F(V_1)^{\sigma_1})=27$ and $h^*(F(V_1)^{\sigma_1}, \IF_3)=67$. By using the formulas
\eqref{fp} and \eqref{chi} we get $m=11$, $a=1$ so looking in the table
we have $S_{\sigma_1}(F(V_1))=U^{\oplus 2}\oplus E_8^{\oplus 2}\oplus A_2$ and $T_{\sigma_1}(F(V_1))=\langle 6\rangle$. Finally observe that the dimension of the family $F(V_1)$ is 10, which is also $\rank(S_{\sigma_1}(F(V_1)))/2-1$ . 
\end{example}

\begin{example}\label{fanofixed}(Case $T=U\oplus A_2^{\oplus 5}\oplus \langle -2 \rangle$). 
Consider the automorphism of order $3$ of $\IP^5$ given by:
$$
\sigma_2(x_0:x_1:x_2:x_3:x_4:x_5)=(x_0:x_1:x_2:x_3:\xi x_4:\xi x_5).
$$
The family of invariant cubics is 
$$
V_2\colon L_3(x_0,x_1,x_2)+M_3(x_4,x_5)=0
$$
where $L_3$ and $M_3$ are homogeneous polynomials of degree $3$. The fixed locus of 
$\sigma_2$ on $V_2$ is  $\{x_0=x_1=x_2=x_3=0, M_3(x_4,x_5)=0\}$ which are three distinct points $p_1,p_2,p_3$ 
and the cubic surface $K$ of $\IP^3$ given by $\{x_4=x_5=0,L_3(x_0,x_1,x_2)=0\}$. An invariant line
through $2$ points must contain also the third point and in fact it is the line
$\{x_0=x_1=x_2=x_3=0\}$, which is not contained in $V_2$. An invariant line through  $p_i$ and a point
of $K$ is contained in $V_2$. So on $F(V_2)$ we have three fixed surfaces isomorphic to the rational
cubic $K$, this has $h^2(K,\IZ)=h^{1,1}(K)=7$. 
Moreover each fixed line on $K$ determine a fixed point on $F(V_2)$ so we have $27$ isolated fixed 
points. By Lemma \ref{actioncubic} the induced automorphism on $F(V_2)$, that we call again $\sigma_2$,  
is non-symplectic. Using the fact that the odd cohomology of the fixed locus is zero we have:
$$
\chi(X^{\sigma_2})=h^*(X^{\sigma_2},\IF_3)=3(2+7)+27=54.
$$ 
Then one computes that $m=a=5$ by using the formulas \eqref{chi} and \eqref{fp}.
 Looking in the table we have 
$S_{\sigma_2}(F(V_2))=U\oplus U(3)\oplus A_2^{\oplus 3}$ and 
$T_{\sigma_2}(F(V_2))=U\oplus A_2^{\oplus 5}\oplus \langle -2\rangle $. Finally the dimension 
of the family is 4 which is equal to $\rank(S_{\sigma_2}(F(V_2)))/2-1$. 
 Observe that
 this automorphism does not have the same fixed locus as the natural automorphism
 on a Hilbert scheme with the same lattices $S$ and $T$. 
\end{example}

\begin{example}\label{fanononnat2}
(Case $T=\langle 6\rangle\oplus E_6^\vee(3)$). Consider the automorphism of order $3$ of $\IP^5$ given by:
$$
\sigma_3(x_0:x_1:x_2:x_3:x_4:x_5)=(x_0:x_1:x_2:\xi x_3:\xi x_4:\xi^2 x_5).
$$
The family of invariant cubics is 
$$
V_3\colon L_3(x_0,x_1,x_2)+M_3(x_3,x_4)+x_5^3+x_5(x_3L_1(x_0,x_1,x_2)+x_4M_1(x_0,x_1,x_2))=0
$$
where $L_3$ and $M_3$ are homogeneous polynomials of degree 3 and $L_1$ and $M_1$
are linear forms. The fixed locus of $\sigma_3$ on $V_3$ is the union of the 
 elliptic curve $E:\{x_3=x_4=x_5=0,L_3(x_0,x_1,x_2)=0\}$ and of the $3$ isolated 
 fixed points $q_1,q_2,q_3$ given by $\{x_0=x_1=x_2=x_5=0,M_3(x_3,x_4)=0\}$. Any invariant line containing 
 two fixed points on $E$ is contained in the plane $\{x_3=x_4=x_5=0\}$ so it cannot be
 contained on $V_3$. On the other hand a line through two of the points $q_j$ must contain also the third 
 point and it is the line $\{x_3=x_4=x_5=0\}$ which is not contained in $V_3$. 
 Finally all invariant lines through a point $q_j$, $j=1,2,3$ and a point of $E$ are contained in 
 $V_3$ so $F(V_3)^{\sigma_3}$ containes $3$ elliptic curves isomorphic to $E$. 
 By Lemma \ref{actioncubic}  the induced automorphism on $F(V_3)$ is non-symplectic. 
 One computes then $\chi(X^{\sigma_3})=0$ and $h^*(X^{\sigma_3},\IF_3)=12$ so by using the formulas \eqref{fp} and \eqref{chi}
 we get $m=8, a=6$. By looking in the Table \ref{ord3} one finds that $S_{\sigma_3}(F(V_3))=U^{\oplus 2}\oplus A_2^{\oplus 6}$,
 $T_{\sigma_3}(F(V_3))=\langle 6\rangle \oplus E_6^\vee(3)$. Finally the family of varieties 
 $F(V_3)$ is $7$-dimensional which is equal to $\rank(S_{\sigma_3}(F(V_3)))/2-1$.
\end{example}

\begin{example}\label{fanononnat3}(Case $T=U(3)\oplus E_6^\vee(3)\oplus \langle -2\rangle$). 
Consider the automorphism of order $3$ of $\IP^5$ given by:
$$
\sigma_4(x_0:x_1:x_2:x_3:x_4:x_5)= (x_0:x_1:\xi x_2:\xi x_3:\xi^2 x_4:\xi^2 x_5).
$$
The family of (projective) invariant cubics is 
\begin{align*}
V_4&: x_2L_2(x_0,x_1)+x_3M_2(x_0,x_1)+x_4^2L_1(x_0,x_1)+x_4x_5M_1(x_0,x_1)\\
&\quad +x_5^2N_1(x_0,x_1)+x_4N_2(x_2,x_3)+x_5P_2(x_2,x_3)=0
\end{align*}
where $L_2$, $M_2$, $N_2$ and $P_2$ are homogeneous polynomial of degree 2 and $L_1$, $M_1$ and $N_1$ are
linear factors.  The fixed locus of $\sigma_4$ on $V_4$ is the union of the three projective lines
$L_1$, $L_2$, $L_3$ of equations $\{x_0=x_1=x_2=x_3=0\}$, $\{x_0=x_1=x_4=x_5=0\}$, $\{x_2=x_3=x_4=x_5=0\}$.
Each line determines a fixed point on $F(V_4)$. On the other hand an invariant line $L$ intersects two of
the lines $L_i$ in one point each. Take a line $L$ intersecting for example the lines $L_1$ and $L_2$ 
in the point $(0:0:0:0:p_4:p_5)$ respectively $(0:0:q_2:q_3:0:0)$, then this points satisfy
the equation $p_4N_2(q_2,q_3)+p_5P_2(q_2,q_3)=0$. This is a rational curve on the surface
$L_1\times L_2\cong \IP^1\times \IP^1$, hence we have a fixed rational curve on $F(V_4)$. In the
same way taking the lines $L_1,L_3$ and $L_2,L_3$ we get $2$ rational fixed curves  on $F(V_4)$. 
By Lemma \ref{actioncubic} $\sigma_4$ acts non-symplectically on $F(V_4)$. One computes $\chi(F(V_4)^{\sigma_4})=9$ and
$h^*( F(V_4)^{\sigma_4},\IF_3)=9$, checking in the Table \ref{ord3} one finds that  $S_{\sigma_4}(F(V_4))=U\oplus U(3)\oplus A_2^{\oplus 5}$,
 $T_{\sigma_4}(F(V_4))=U(3) \oplus E_6^\vee(3)\oplus \langle -2\rangle$. Finally the family of varieties 
 $F(V_4)$ is $6$-dimensional which is equal to $\rank(S_{\sigma_4}(F(V_4)))/2-1$. 
 Observe that this automorphism has the same fixed locus as the natural automorphism
 on a Hilbert scheme with the same lattices $S$ and $T$. 
\end{example}

In the tables \ref{ord3}, \ref{ord7}, \ref{ord11}, \ref{ord13}, \ref{ord17}, \ref{ord19} we mark with a  $\diamondsuit$ the cases that are realised with automorphisms on the Fano variety of lines of a cubic fourfold.

\begin{rem}
Let $ V $ be a 6-dimensional vector space.  The wedge product $ \wedge:\wedge^3V\times\wedge^3V
\longrightarrow \wedge^6V  $ induces a sympletic form $ \omega$ on $ \wedge^3V $ by choosing an isomorphism $ \wedge^6V\cong\mathbb{C} $. Let us take a Lagrangian
subspace $ A \subset \wedge^3V$ and let us define $Y_A:=\left\lbrace v\in \IP (V)/(v\wedge\wedge^2 V)\cap A\neq 0\right\rbrace  $; for general $ A
$, $ Y_A $ is a hypersurface  of degree 6 of the type described by Eisenbud--Popescu--Walter. Such a hypersurface is not smooth, but for a general $A$ it has a smooth double cover $ X_A $ which is an \ihsk (see \cite{OG}). 
One can construct automorphisms of $X_A$ induced by automorphisms of prime order of $V$ as done in \cite{Camere} and in \cite{MongardiPhD}. Direct computations show that
using this method one can only construct non-symplectic automorphisms of order two, and these examples are already explained in \loccit
\end{rem}

\section{Existence of automorphisms}

Starting from our list of all admissible values of $(p,m,a)$ we want to realize each
case by an automorphism, using Theorem~\ref{ball} and the examples of Section~\ref{s:exampl}. The following result shows that it is always possible, except in one case:

\begin{theorem}\label{exi}
For every admissible value of $(p,m,a)\neq(13,1,0)$ there exists an \ihsk  with a non-symplectic automorphism $\sigma$ of order $p$ whose invariant lattice $T_{\sigma}$ and orthogonal lattice $S_{\sigma}$ are those characterized in Theorem~\ref{th:lattice}. The fixed locus is a disjoint union of isolated fixed points, smooth curves and smooth surfaces whose invariants $h^\ast$ and $\chi$ are
given in formulas \eqref{fp} and \eqref{chi}.
\end{theorem}

\begin{proof}  We refer to Tables  \ref{ord3}, \ref{ord7}, \ref{ord11}, \ref{ord13}, \ref{ord17}, \ref{ord19}  in Appendix~\ref{app:tables} for the isometry class
of the lattices and the values of $\chi$ and $h^*$, that depend only on $(p,m,a)$.

\par{\it Existence.} Observe that by Proposition \ref{unicityT} the lattice $T$ has a unique primitive embedding in $L$. So except for $(p,m,a)=(3,11,1)$ and $(p,m,a)=(3,8,6)$ it is not a restriction to assume that $T=\overline{T}\oplus \langle -2\rangle$ with $\overline{T}$ primitively embedded in $\Lambda$. Moreover by Theorem~\ref{indefigen} the lattice $S$ is uniquely determined (the case $S=A_2(-1)$ is a direct computation) and it is in fact the orthogonal complement of $\overline{T}$ in $\Lambda$. 
Assume that $T=\overline{T}\oplus \langle-2\rangle$. If $\rank S>p-1$ the existence follows 
from Theorem~\ref{ball}, the moduli space is a complex ball of dimension $\dim S(\xi)-1$ and these cases are realized by  natural automorphisms on
 the Hilbert scheme of points of a K3 surface (see  \S~\ref{natural} for explicit examples). 
If $\rank S=p-1$ then $\dim\Omega_T^{\rho,\xi}=0$ and the existence follows again from the existence of a K3 surface
with a non-symplectic automorphism of order $p$ with lattice $\overline{T}$ and $S$.
The remaining cases $(3,11,1)$ and $(3,8,6)$ are realized by automorphisms on Fano varieties of lines on cubic fourfolds in Example~\ref{fanononnat1} and Example~\ref{fanononnat2}..

\par{\it Fixed locus.} It is clear that the fixed locus is smooth, we prove that the maximal dimension of a fixed component is two. By the Local Inversion Theorem, since $\sigma$ has finite order its action at the neighbourhood of a fixed point $x$ can be linearized as an action of a 4$\times$4-matrix~$M$. The dimension of the fixed locus at $x$ is the multiplicity of $1$ as an eigenvalue of $M$. It follows that the fixed locus is the disjoint union of smooth connected subvarieties. Since $\sigma$ is non-symplectic, one has $M^TJM=\xi_p J$ where $J=\left(\begin{matrix} 0& I_2\\ -I_2& 0\end{matrix}\right)$ is the standard symplectic form and $\xi_p$ is some primitive $p$-th root of the unity. Since $M^{-1}=\xi_p^{-1}J^{-1}M^TJ$ one observes that if $\lambda$ is an eigenvalue of $M$, then $\xi_p/\lambda$ is also an eigenvalue of $M$ with the same multiplicity. It follows that there are two possible sequences of eigenvalues for $M$: $(1,\xi_p,\xi_p^{(p+1)/2})$ with multiplicities $(a,a,b)$ or $(1,\xi_p,\xi_p^i,\xi_p^{-i+1})$ with $2i\not\equiv 1\mod (p)$ with multiplicities $(a,a,b,b)$. Since the sum of the multiplicites equals $4$, we conclude that $a\leq 2$ so the maximal dimension of a fixed component is two.
\end{proof}

\begin{rem}
One can get some extra information on the local action of an automorphism $\sigma$ at the neighbourhood 
of a fixed point
by using the Pfaffian. With the same notation as in the proof above, one has:
$$
\xi_p^2=\Pf(\xi_p J)=\Pf(M^TJM)=\Pf(J)\det(M)=\det(M).
$$
In the first case mentioned in the proof, the eigenvalues are $(1,\xi_p,\xi_p^{(p+1)/2})$ with multiplicities $(a,a,b)$
such that $2a+b=4$ and the equation above gives $a+\frac{p+1}{2}b\equiv 2\mod (p)$. In the second case, the eigenvalues are $(1,\xi_p,\xi_p^i,\xi_p^{-i+1})$ with $2i\not\equiv 1\mod (p)$ with multiplicities $(a,a,b,b)$ such that $a+b=2$.
\end{rem} 

\begin{rem}
The  case $(13,1,0)$ can not be realized by a natural automorphism for the following reason. Suppose that there exists an 
holomorphic symplectic manifold $X$ with a natural
 automorphism $\sigma$ with invariants $(p,m,a)=(13,1,0)$, $S=U\oplus U\oplus E_8$ 
and invariant lattice $T=U\oplus E_8\oplus \langle -2\rangle$.
If this automorphism is natural then there exists  a K3 surface with a non-symplectic automorphism 
of order $13$ and $\overline{T}=U\oplus E_8$, 
$\bar{S}=U\oplus U\oplus E_8$, but by \cite[Theorem 4.3]{Kondo} such a K3 surface does not exist.
This situation is similar to the case $p=23$  (the maximum prime order for a non-symplectic
automorphism on such varieties, see \cite[\S~5.4]{BNWS}) that cannot be realized by a natural automorphism.
It is an open problem to construct (or exclude) such non-natural automorphisms  of order $13$ or $23$ on \ihskpt
\end{rem}

\begin{rem}
In the case $p=2$ one can do a similar construction as above, but the situation is more complicated. E.g. we can have several non equivalent embeddings of $T$ with different orthogonal complements $S$, which give then different moduli spaces (see Proposition~\ref{2-elem} below and \cite{Malek} for a description of the moduli spaces).
\end{rem}

\begin{cor}
Let $X$ be an \ihsk with a non-symplectic automorphism $\sigma$ of prime order $3\leq p\leq 19$, $p\not=5$. If
$(p,m,a)\notin\{(3,8,6),(3,11,1),(13,1,0)\}$ then the action is unique in the sense of Corollary \ref{unicity}. 
\end{cor}
\begin{proof} 
By assumption, using Theorem~\ref{th:lattice} in each case the lattice $T$ is  of the form $\overline{T}\oplus \langle -2\rangle$ so the assertion follows from Theorem~\ref{unicity} and from the fact that the cases where the moduli space is zero-dimensional are realized by natural automorphisms on K3 surfaces.
\end{proof}

\begin{cor}\label{cor:nonnatural}
There exists \ihsk with a non-natural non-symplectic automorphisms of prime order $p\geq 3$.
\end{cor}

\begin{proof}
We provide two examples of \ihsk admitting a non-symplectic automorphism that can not be the deformation of an automorphism on a Hilbert scheme of points induced
by an automorphism of the underlying K3 surface.
\par $(p,m,a)=(3,11,1)$. This is Example \ref{fanononnat1}, here $\rank T=2$.
For every automorphism on some $\Sigma^{[2]}$ induced by an automorphim of the K3 surface $\Sigma$, the rank of the invariant lattice is at least two, since its contains the class of an ample divisor on the K3 surface and the class of the exceptional divisor is also invariant. Since~$T$ is invariant under 
equivariant deformation this automorphism is non-natural.

\par $(p,m,a)=(3,8,6)$. This is Example \ref{fanononnat2}, here $S=U^{\oplus 2}\oplus A_2^{\oplus 6}$. This lattice is invariant under equivariant deformation. Checking all
automorphisms obtained using K3 surfaces one sees that this lattice can not be obtained in this way, so this automorphism is non-natural. 
\end{proof}

\begin{rem}\label{rem:fixnotunique} {\bf(Different fixed loci).}
We have shown that in many cases the action of the automorphism on the lattice $L$ is uniquely determined (see Corollary~\ref{unicity}), but in general the fixed locus is not uniquely determined, for instance in the case $(p,m,a)=(3,5,5)$. As explained in \S~\ref{natural}, starting with a K3 surface $\Sigma$ with a non-symplectic automorphism of order three with fixed locus consisting of $5$ isolated fixed points and two rational curves (see \cite[Table 2]{AST}) we obtain a natural automorphism on the Hilbert scheme $\Sigma^{[2]}$ with fixed locus consisting in $10$ isolated fixed points, $17$~rational curves, one surface isomorphic to $\IP^1\times \IP^1$ and two surfaces isomorphic to $\IP^2$. As explained in Example \ref{fanofixed}, using the Fano variety of lines on a cubic fourfold we construct a non-symplectic automorphism with a  different fixed locus consisting
 this time of $27$ isolated fixed points and three rational cubic surfaces. 
In contrary, Example \ref{fanononnat3} shows a similar situation where the fixed loci are identical.
\end{rem}

%%%%%%%%%%%%%%%%%%%%%%%%%%%%%%%%%%%%%%%%%%%%%%%%%%%%%%%

\section{Non-symplectic Involutions}\label{esempi5}

Beauville in \cite[Proposition 2.2]{BeauvilleInv} shows that the fixed locus of a non-symplectic involution $\sigma$ on an \ihsk is a smooth Lagrangian surface
$F$ (possibly not connected) such that $\chi(\cO_F)=\frac{1}{8}(t^2+7)$ and $e(F)=\frac{1}{2}(t^2+23)$, where $t$ is the trace of $\sigma^*$ on
$H^{1,1}(X)$, and he proves that $t$ can take all odd integer values $-19\leq t\leq 21$. Moreover, he provides examples for all such cases: these are all natural examples except one:  Beauville's non-natural example from \cite{BeauvilleKaehler} for $t=-19$. Nevertheless, the paper contains no information about the invariant
lattice and its orthogonal.

Ohashi and Wandel in \cite{OW} study the case $t=-17$ in detail by classifying all possible conjugacy classes of non-symplectic involutions. In fact, such conjugacy classes are in bijection with the orbits of primitive embeddings of the invariant sublattice $T$ in $L$. Moreover, in their paper they show that all conjugacy classes are indeed realized, at least abstractly, and give a new explicit example for one of the families using moduli spaces of sheaves on K3 surfaces.  

The proof of  \cite[Lemma 5.5]{BNWS}, with little modification, gives the following result for the case $p=2$:

\begin{lemma}\label{2elemdiscgrp}
Let $X$ be an \ihsk and $G$ a 
finite group of order $2$ acting non-symplectically on $X$. Then:
\begin{itemize}
\item $\frac{H^2(X,\IZ)}{S\oplus T}\cong \left( \frac{\IZ}{p\IZ}\right)^{\oplus a_G(X)}$ for some integer
$0\leq a_G(X)=:a$. 
\item $S$ has signature $(2,r-2)$ and $T$ has signature $(1, 22-r)$ where $r=\rank S$; 
\item either $A_T\cong (\IZ/2\IZ)^{\oplus a+1}$ and $A_S\cong (\IZ/2\IZ)^{\oplus a}$ or {\it vice versa}.
\end{itemize}
\end{lemma}

Hence both $T$ and $S$ are indefinite $2$-elementary lattices and $T$ is hyperbolic.
 Recall that the isomorphism class of an indefinite $2$-elementary lattice is classified 
by the triple $(r,a,\delta)$, where $r$ is its rank, $a$ is the length of the discriminant 
group and $\delta=0$ if the discriminant form takes values in $\IZ/2\IZ\subset\IQ/2\IZ$ and $1$ otherwise. 
We are interested in counting how many non-isomorphic primitive  embeddings of $T$  in $L$ there are. Proposition~\ref{2-elem} below is the analogous of the Proposition~\ref{discform} in the case of involutions. We formulate it this time for the lattice $T$ instead of $S$ for compatibility with Nikulin's classification \cite{Nikulinfactor} of non-symplectic involutions on K3 surfaces in terms of hyperbolic $2$-elementary lattices.

\begin{prop}\label{2-elem}
Let $T$ be an even hyperbolic $2$-elementary lattice of signature $(1,t)$ and length $\ell(A_T)=a\geq 0$, and  $L=U^{\oplus 3}\oplus E_8^{\oplus
2}\oplus \langle -2\rangle$. Assume that $T$ admits a primitive embedding in $L$. Then:
\begin{itemize}
 \item[(i)] If there is no $x\in A_T$ such that $q_T(x)=3/2 \mod 2\IZ$, then $T$ admits a unique primitive embedding into $L$ whose orthogonal complement is a $2$-elementary
lattice $S$ of signature $(2,20-t)$ length $\ell(A_S)=a+1$ and $\delta_S=1$.
\item[(ii)] Otherwise, non-isomorphic primitive embeddings of $T$ into $L$ are in 1-1 correspondence with
non-isometric choices of a $2$-elementary lattice $S$ of signature $(2,20-t)$ with either $l(A_S)=l(A_T)-1$, or $l(A_S)=l(A_T)+1$ and
$\delta_S=1$.
\end{itemize}
\end{prop}
\begin{proof}
We proceed as in the proof of Proposition \ref{discform}. 
By \cite[Proposition 1.15.1]{Nikulinintegral} a primitive embedding
of $T$ into $L$ is equivalent to the data of a quintuple $(H_T,H_L,\gamma,S,\gamma_S)$ satisfying the following conditions:
\begin{itemize}
\item $H_T$ is a subgroup of $A_T=(\IZ/2\IZ)^{\oplus a}$, $H_L$ is a subgroup of $A_L=\IZ/2\IZ$ and $\gamma\colon H_T\to H_L$ is an isomorphism of groups such that for any $x\in H_T$, $q_L(\gamma(x))=q_T(x)$.
\item $S$ is a lattice of invariants $(2,20-t,q_S)$ with $q_S=\left.\left((-q_T)\oplus q_L\right)\right|_{\Gamma^\perp/\Gamma}$,
where $\Gamma$ is the graph of $\gamma$ in $A_T\oplus A_L$, $\Gamma^{\perp}$ is the orthogonal complement of $\Gamma$ in
$A_T\oplus A_L$ with respect to the bilinear form induced on $A_T\oplus A_L$ and with values in $\IQ/\IZ$ and $\gamma_S$ is an automorphism of $A_S$ that preserves $q_S$. Moreover $S$ is the orthogonal complement of $T$ in $L$. 
\end{itemize}
In our case there are only two possibilities:
\begin{enumerate}
 \item $H_T=H_L=\{0\}$ and $\gamma=\id$. In this case $\Gamma=\{(0,0)\}$, so
the discriminant group of the orthogonal complement is $A_T\oplus A_L$ and $q_S=-q_T\oplus q_L$. Recall that $A_L=\frac{\IZ}{2\IZ}\left(\frac{3}{2}\right)$ so let $y\in A_L$ be such that $q_L(y)=3/2$; then
$q_S((0,y))=q_L(y)\notin\IZ/2\IZ$ and hence $\delta_S=1$.
\item $H_T=H_L=\IZ/2\IZ$ and $\gamma=\id$. This case can happen only if there is $x\in A_T$ such that $q_T(x)=3/2$. In this case
$\Gamma\cong \IZ/2\IZ$, so 
the discriminant group of $S$ is $A_S=\Gamma^{\perp}/\Gamma\cong (\IZ/2\IZ)^{\oplus a-1} $ and $q_S=(-q_T\oplus q_L)_{|A_S}$.
\end{enumerate}
In each case the lattice $S$ is $2$-elementary so the natural map $O(S)\rightarrow O(q_S)$ is surjective by \cite[Theorem 3.6.3]{Nikulinintegral}. This implies that different choices of the isometry~$\gamma_S$ produce isomorphic embeddings of $T$ in $L$. As a consequence, if there is no $x\in A_T$ such that $q_T(x)=3/2\mod 2\IZ$ only the first case occurs and the lattice $S$ has signature $(2,20-t)$, $\ell(A_S)=a+1$ and $\delta_S=1$. As noted in Remark~\ref{rem:2elem}, the lattice~$S$ is uniquely determined by these invariants so $T$ admits a unique
embedding in~$L$. Otherwise, both cases can occur and the primitive embeddings of $T$ in $L$ are classified by $S$: the signature of $S$ is $(2,20-t)$, in the first case $S$ has length $a+1$ and $\delta_S=1$, in the second case 
its length is $a-1$.
\end{proof}

\begin{rem} As a consequence, the lattice $T$ admits at most three non isometric embeddings in $L$. In particular:
\begin{enumerate}
\item If $\delta_T=0$ then case (i) occurs and $T$ admits a unique primitive embedding in $L$. By \cite[Theorem~4.3.2]{Nikulinfactor} this implies that $1-t\equiv 0\mod 4$ so $\rank T\equiv 2\mod 4$.
\item By \cite[Theorem~1.5.2]{dolgabourbaki} if $\delta_S=0$ then $t-18\equiv 0\mod 4$ so $\rank S\equiv 0\mod 4$ and $\rank T\equiv 3\mod 4$.
Hence if $\rank T\not\equiv 3\mod 4$ one has necessarily $\delta_S=1$ hence $T$ admits at most two embeddings in $L$.
\end{enumerate}
\end{rem}

In Figures \ref{app:triangle1},\ref{app:triangle2}  we give all possible values of 
$(r,a,\delta)$ such that a $2$-elementary lattice $T$ with these invariants
admits a primitive embedding in $L$. In order to show all possible embeddings, 
in Figure \ref{app:triangle1} we give the cases where the orthogonal~$S$ has the property $\ell(A_S)=\ell(A_T)+1$ and in Figure \ref{app:triangle2} 
we give the cases where $\ell(A_S)=\ell(A_T)-1$. These figures are obtained by using the results of Nikulin on primivite embeddings (see Remark~\ref{rem:2elem}).

\begin{rem}\label{nat-inv}
Let $\Sigma$ be a K3 surface with a non-symplectic involution $\iota$ such that the invariant 
lattice $T_{\iota}(\Sigma)$ has invariants $(r,a,\delta)$. Then the natural involution 
$\iota^{[2]}$ induced by $\iota$ on $\Sigma^{[2]}$ gives an example of the case where the 
invariant sublattice~$T$ has invariants $(r+1,a+1,1)$, its orthogonal complement $S$ 
satisfies $\ell(A_S)=a=\ell(A_T)-1$ and $\delta_S=\delta$. This gives a realization of all the cases 
illustrated in Figure~\ref{app:triangle2}. 
\end{rem}

\begin{prop}\label{embedT}
Under the same assumptions as in Proposition \ref{2-elem}, for each embedding $j:T\hookrightarrow L$ 
there exists an \ihsk with a non-symplectic involution $\sigma\colon X\to X$ such 
that the invariant lattice $T_\sigma(X)\subset H^2(X,\IZ)$ is isomorphic to the embedding $j(T)\subset L$.
\end{prop}

\begin{proof}
The proof follows the same lines as those of \cite[Lemma 2.6]{OW} with one exception. 
The isometry $i=\id_T\oplus (-\id_S)$ of $T\oplus S$ induces 
the identity on $A_{S\oplus T}$, so it leaves stable 
the subgroup $\frac{L}{T\oplus S}\subset A_{T\oplus S}$. 
This implies that $i$ extends to an isometry of $L$ such that $L^i=T$ 
(see \cite[Corollary~1.5.2]{Nikulinintegral}).
Assume first that $\rank T\leq 20$. By the surjectivity of the period map $P_0$ and Proposition~\ref{prop:dense},
for any $\omega\in\Omega_T^\circ$ there exists an \ihsk
with a marking $\eta\colon L\to H^2(X,\IZ)$ such that $\eta(T)=\NS(X)$ and $\eta(\omega)=H^{2,0}(X)$.
Then $\NS(X)$ is hyperbolic, so $X$ is projective by \cite[Theorem~3.11]{Huybrechts}. The action of $i$
on $H^2(X,\IZ)$ induced by $\eta$, that we still denote by $i$, is a Hodge isometry since $i(\omega)=-\omega$ implies $i(H^{2,0}(X))=H^{2,0}(X)$. We apply the Strong Torelli Theorem as stated in \cite[Theorem~1.3]{Markmantorelli}. We verify the conditions of \cite[Theorem~9.8]{Markmantorelli} to prove that the involution $i$ is a parallel transport operator. The positive vectors of~$T$ and~$S$ generate a positive three dimensional subspace of~$L_\IR$ 
whose orientation is preserved by $i$. Moreover since the lattice $L$ admits a unique primitive embedding in the Mukai lattice $U^{\oplus 4}\oplus E_8^{\oplus 2}$, we conclude that the involution $i$ is a parallel 
transport operator. Since $X$ is projective and $i$ acts trivially on $\NS(X)$, it leaves invariant an ample class, so $i$ maps the K\"ahler cone of $X$ to itself. By the Strong Torelli Theorem it follows that there exists an automorphism $\iota$ of $X$ such that $\iota^\ast=i$. Since the natural 
map $\Aut(X)\to O(H^2(X,\IZ))$ is injective, $\iota$ is an involution.

Assume now that $\rank T=21$: the previous argument does not work since the period domain is zero-dimensional but we observe that by \cite{Nikulinfactor} there exists a  K3 surface $\Sigma$ with an involution $\iota$ such that the invariant lattice $T_{\iota}(\Sigma)$ has invariants $(20,2,1)$. Then by Remark~\ref{nat-inv} the involution $\iota^{[2]}$ on $\Sigma^{[2]}$ is such that $(\iota^{[2]})^\ast=i$ and realizes the case where the invariant sublattice $T$ has invariants $(21,3,1)$.
\end{proof}

\begin{example}\text{}
\begin{enumerate}
 \item Take $T=\langle 2 \rangle$, of invariants $(1,1,1)$. The unique embedding in $L$ (see Figure~\ref{app:triangle1}) corresponds to Beauville's
non-natural involution \cite{BeauvilleKaehler} on the Hilbert scheme of two points of a quartic in $\IP^4$ containing no line, here $S=U^{\oplus 2}\oplus E_8^{\oplus 2}\oplus \langle -2\rangle^{\oplus 2}$.
\item Take $T=\langle 2\rangle\oplus \langle -2\rangle$ of invariants $(2,2,1)$. The embedding in Figure~\ref{app:triangle1}
has orthogonal $S=U^{\oplus 2}\oplus E_8\oplus E_7\oplus \langle -2\rangle^{\oplus 2}$, it corresponds to
Ohashi--Wandel's involution \cite{OW}. The embedding in Figure~\ref{app:triangle2} has orthogonal
$S=U^{\oplus 2}\oplus E_8^{\oplus 2}\oplus \langle -2\rangle$ and is realized by a natural involution
on the Hilbert scheme of two points on a K3 surface (see Remark~\ref{nat-inv}).
\end{enumerate}
\end{example}

\clearpage

\appendix

\section{Tables for the invariant lattice and its orthogonal}
\label{app:tables}

\begin{table}[!ht]
\begin{tabular}{r|c|c|c|c|c|c}
$p$&$m$&$a$&$\chi$&$h^*$&$S$&$T$\\
\hline
$\diamondsuit$ 3&11&1&27&67&$U^{\oplus 2}\oplus E_8^{\oplus 2}\oplus A_2$&$\langle 6\rangle$\\

\hline
$\clubsuit$ 3&10&0&9&109&$U^{\oplus 2}\oplus E_8^{\oplus 2}$&$U\oplus \langle -2\rangle$\\

$\clubsuit$ 3&10&2&9&57&$U\oplus U(3)\oplus E_8^{\oplus 2}$&$U(3)\oplus\langle -2\rangle$\\
\hline

$\clubsuit$ 3&9&1&0&96&$U^{\oplus 2}\oplus E_6\oplus E_8$&$U\oplus A_2\oplus \langle -2\rangle$\\

$\clubsuit$ 3&9&3&0&48&$U\oplus U(3)\oplus E_6\oplus E_8$&$U(3)\oplus A_2\oplus \langle -2\rangle$\\

\hline

$\clubsuit$ 3&8&2&0&84&$U^{\oplus 2}\oplus E_6\oplus E_6$&$U\oplus A_2^{\oplus 2}\oplus \langle -2\rangle$\\

$\clubsuit$ 3&8&4&0&40&$U\oplus U(3)\oplus E_6^{\oplus 2}$&$U(3)\oplus A_2^{\oplus 2}\oplus \langle -2\rangle$\\

$\diamondsuit$ 3&8&6&0&12&$U^{\oplus 2}\oplus A_2^6$&$\langle 6\rangle\oplus E_6^\vee(3)$\\
\hline

$\clubsuit$ 3&7&1&9&129&$ U^{\oplus 2}\oplus A_2\oplus E_8 $&$U\oplus E_6\oplus \langle -2\rangle$\\

$\clubsuit$ 3&7&3&9&73& $U\oplus U(3)\oplus A_2\oplus E_8$&$U\oplus A_2^{\oplus 3}\oplus \langle -2\rangle$\\

$\clubsuit$ 3&7&5&9&33&$U^{\oplus 2}\oplus A_2^5$&$U(3)\oplus A_2^3\oplus \langle -2\rangle$\\

$\diamondsuit$, $\clubsuit$ 3&7&7&9&9&$U\oplus U(3)\oplus A_2^5$&$U(3)\oplus E_6^\vee(3)\oplus \langle -2\rangle$\\
\hline
$\clubsuit$ 3&6&0&27&183&$U^{\oplus 2}\oplus E_8$&$U\oplus E_8\oplus \langle -2\rangle$\\

$\clubsuit$ 3&6&2&27&115&$U\oplus U(3)\oplus E_8$&$U\oplus E_6\oplus A_2\oplus \langle -2\rangle$\\

$\clubsuit$ 3&6&4&27&63&$U^{\oplus 2}\oplus A_2^4$&$U\oplus A_2^4\oplus \langle -2\rangle$\\

$\clubsuit$ 3&6&6&27&27&$U\oplus U(3)\oplus A_2^4$&$U(3)\oplus A_2^4\oplus \langle -2\rangle$\\
\hline

$\clubsuit$ 3&5&1&54&166&$U^{\oplus 2}\oplus E_6$&$U\oplus E_8\oplus A_2\oplus \langle -2\rangle$\\

$\clubsuit$ 3&5&3&54&102&$U\oplus U(3)\oplus E_6$&$U\oplus A_2^2\oplus E_6\oplus \langle -2\rangle$\\

$\diamondsuit$, \,$\clubsuit$ 3&5&5&54&54&$U\oplus U(3)\oplus A_2^3$&$U\oplus A_2^5\oplus \langle -2\rangle$\\
\hline

$\clubsuit$ 3&4&2&90&150&$U^{\oplus 2}\oplus A_2^2$&$U\oplus E_6^2\oplus \langle -2\rangle$\\

$\clubsuit$ 3&4&4&90&90&$U\oplus U(3)\oplus A_2^2$&$U\oplus E_6\oplus A_2^3\oplus \langle -2\rangle$\\

\hline

$\clubsuit$ 3&3&1&135&207&$U^{\oplus 2}\oplus A_2$&$U\oplus E_6\oplus E_8\oplus \langle -2\rangle$\\

$\clubsuit$ 3&3&3&135&135&$U\oplus U(3)\oplus A_2$&$U\oplus E_6^{\oplus 2}\oplus A_2 \oplus \langle -2\rangle$\\
\hline
$\clubsuit$ 3&2&0&189&273&$U^{\oplus 2}$&$U\oplus E_8^{\oplus 2}\oplus \langle -2\rangle$\\

$\clubsuit$ 3&2&2&189&189&$U\oplus U(3)$&$U\oplus E_6\oplus E_8\oplus A_2\oplus \langle -2\rangle$\\
\hline

$\clubsuit$ 3&1&1&252&252&$A_2(-1)$&$U\oplus E_8^{\oplus 2}\oplus A_2\oplus \langle -2\rangle$\\
\end{tabular}
\caption{Order 3}\label{ord3}
\end{table}
\bigskip
\begin{table}[!ht]
\begin{tabular}{c|c|c|c|c|c}
$m$&$a$&$\chi(X^\sigma)$&$h^*(X^G,\IF_p)$&$S_G(X)$&$T_G(X)$\\
\hline

 5&1&-1&31&$U\oplus E_8^{\oplus 2}\oplus H_5$&$H_5 \oplus \langle -2\rangle$\\
\hline
 4&2&14&42&$U\oplus H_5\oplus E_8\oplus A_4$&$ H_5\oplus A_4\oplus \langle -2\rangle$\\

 4&4&14&14&$U(5)\oplus H_5\oplus E_8\oplus A_4$&$H_5\oplus A_4^*(5)\oplus \langle -2\rangle$\\
\hline
3&1&54&102&$U\oplus H_5\oplus E_8$&$ H_5\oplus E_8\oplus \langle -2\rangle$\\

 3&3&54&54&$U\oplus H_5\oplus A_4^2$&$ H_5\oplus A_4^2\oplus \langle -2\rangle$\\
\hline
 2&2&119&119&$U\oplus H_5\oplus A_4$&$H_5\oplus A_4\oplus E_8\oplus \langle -2\rangle$\\
\hline
 1&1&202&202&$U\oplus H_5$&$ H_5\oplus E_8^2\oplus \langle -2\rangle$\\
\end{tabular}
\caption{Order 5}\label{ord5}
\end{table}
\bigskip

\begin{table}
\begin{tabular}{c|c|c|c|c|c|c}
$p$&$m$&$a$&$\chi$&$h^*$&$S$&$T$\\
\hline

$\clubsuit$ 7&3&1&9&33&$U^{\oplus 2}\oplus E_8\oplus A_6$&$ U\oplus K_7\oplus \langle -2\rangle$\\

$\clubsuit$ 7&3&3&9&9&$U\oplus U(7)\oplus E_8\oplus A_6$&$ U(7)\oplus K_7\oplus \langle -2\rangle$\\
\hline
$\clubsuit$ 7&2&0&65&117&$U^{\oplus 2}\oplus E_8$&$U\oplus E_8\oplus \langle -2\rangle$\\

$\clubsuit$ 7&2&2&65&65&$U\oplus U(7)\oplus E_8$&$U(7)\oplus E_8\oplus \langle -2\rangle$\\
\hline

$\clubsuit$ 7&1&1&170&170&$U^{\oplus 2}\oplus K_7$&$ U\oplus E_8\oplus A_6\oplus \langle -2\rangle$\\
\end{tabular}
\caption{Order 7}\label{ord7}
\end{table}

\bigskip
 
\begin{table}
\begin{tabular}{r|c|c|c|c|c|c}
$p$&$m$&$a$&$\chi$&$h^*$&$S$&$T$\\
\hline
$\clubsuit$ 11&2&0&5&25&$U^{\oplus 2}\oplus E_8^{\oplus 2}$&$U\oplus \langle -2\rangle$\\

$\clubsuit$ 11&2&2&5&5&$U\oplus U(11)\oplus E_8^{\oplus 2}$&$U(11)\oplus \langle -2\rangle$\\
\hline

$\clubsuit$ 11&1&1&104&104&$K_{11}(-1)\oplus E_8$&$ U\oplus A_{10}\oplus \langle -2\rangle$\\
\end{tabular}
\caption{Order 11}\label{ord11}
\end{table}

\bigskip

\begin{table}
\begin{tabular}{r|c|c|c|c|c|c}
$p$&$m$&$a$&$\chi$&$h^*$&$S$&$T$\\
\hline
13&1&0&77&103&$U^{\oplus 2}\oplus E_8$&$ U\oplus E_8\oplus \langle -2\rangle$\\
$\clubsuit$ 13&1&1&77&77&$U\oplus E_8\oplus H_{13}$&$E_8\oplus H_{13}\oplus \langle -2\rangle$\\
\end{tabular}
\caption{Order 13}\label{ord13}
 \end{table}

\bigskip

\begin{table}
\begin{tabular}{r|c|c|c|c|c|c}
$p$&$m$&$a$&$\chi$&$h^*$&$S$&$T$\\
\hline
$\clubsuit$ 17&1&1&35&35&$U^{\oplus 2}\oplus E_8\oplus L_{17}$&$U\oplus L_{17}\oplus \langle -2\rangle$\\
\end{tabular}
\caption{Order 17}\label{ord17}
\end{table}

\bigskip

\begin{table}
\begin{tabular}{r|c|c|c|c|c|c}
$p$&$m$&$a$&$\chi$&$h^*$&$S$&$T$\\
\hline
$\clubsuit$ 19&1&1&20&20&$K_{19}(-1)\oplus E_8^{\oplus 2}$&$U\oplus K_{19}\oplus \langle -2\rangle$\\
\end{tabular}
\caption{Order 19}\label{ord19}
\end{table}

\begin{figure}[!ht]
\hfill $\begin{array}{l} \sbt \quad \delta_T=\delta_S=1\\ 
\ast \quad \delta_T=0,\, \delta_S=1\end{array}$
\vspace{-1.1cm}

\begin{tikzpicture}[scale=.43]
\filldraw [black] 
(1,1) circle (1.5pt)  %node[below=-0.55cm]{10}
(2,0) node[below=-0.20cm]{*} 
(2,2) circle (1.5pt) node[below=-0.15cm]{*}%node[below=-0.5cm]{9} 
 (3,1) circle (1.5pt)
 (3,3) circle (1.5pt)%node[below=-0.5cm]{8}
 (4,2) circle (1.5pt)
(4,4) circle (1.5pt)%node[below=-0.5cm]{7}
(5,3) circle (1.5pt)
(5,5) circle (1.5pt)%node[below=-0.5cm]{6}
(6,4) circle (1.5pt)node[below=-0.15cm]{*}
(6,2)    node[below=-0.23cm]{*}
(6,6) circle (1.5pt)%node[below=-0.5cm]{5}
(7,3) circle (1.5pt)
(7,5) circle (1.5pt)
(7,7) circle (1.5pt)%node[below=-0.5cm]{4}
(8,2) circle (1.5pt)
(8,4) circle (1.5pt)
(8,6) circle (1.5pt)
(8,8) circle (1.5pt)%node[below=-0.5cm]{3}
(9,1) circle (1.5pt)
(9,3) circle (1.5pt)
(9,5) circle (1.5pt)
(9,7) circle (1.5pt)
(9,9) circle (1.5pt)%node[below=-0.5cm]{2}
(10,0)  node[below=-0.20cm]{*}
(10,2) circle (1.5pt)node[below=-0.15cm]{*}
(10,4) circle (1.5pt)node[below=-0.15cm]{*}
(10,6) circle (1.5pt)node[below=-0.15cm]{*}
(10,8) circle (1.5pt)node[below=-0.15cm]{*}
(10,10) circle (1.5pt)node[below=-0.15cm]{*}
(11,1) circle (1.5pt)
(11,3) circle (1.5pt)
(11,5) circle (1.5pt)
(11,7) circle (1.5pt)
(11,9) circle (1.5pt)
(11,11) circle (1.5pt) %node[below=-0.5cm]{0}
(12,2) circle (1.5pt)
(12,4) circle (1.5pt)
(12,6) circle (1.5pt)
(12,8) circle (1.5pt)
(12,10) circle (1.5pt) %node[below=-0.5cm]{1}
(13,3) circle (1.5pt)
(13,5) circle (1.5pt)
(13,7) circle (1.5pt)
(13,9) circle (1.5pt) %node[below=-0.5cm]{2}
(14,2)  node[below=-0.15cm]{*}
(14,4) circle (1.5pt)node[below=-0.15cm]{*}
(14,6) circle (1.5pt)node[below=-0.15cm]{*}
(14,8) circle (1.5pt) %node[below=-0.5cm]{3}
(15,3) circle (1.5pt)
(15,5) circle (1.5pt)
(15,7) circle (1.5pt)% node[below=-0.5cm]{4}
(16,2) circle (1.5pt)
(16,4) circle (1.5pt)
(16,6) circle (1.5pt) %node[below=-0.5cm]{5}
(17,1) circle (1.5pt)
(17,3) circle (1.5pt)
(17,5) circle (1.5pt)%node[below=-0.5cm]{6}
(18,0)  node[below=-0.2cm]{*}
(18,2) circle (1.5pt)node[below=-0.15cm]{*}
(18,4) circle (1.5pt) node[below=-0.15cm]{*}
(19,1) circle (1.5pt)
(19,3) circle (1.5pt) %node[below=-0.5cm]{8}
(20,2) circle (1.5pt) %node[below=-0.5cm]{9}
 ; 
\draw[->] (0,0) -- coordinate (x axis mid) (22,0);
    \draw[->] (0,0) -- coordinate (y axis mid)(0,13);
    \foreach \x in {0,1,2,3,4,5,6,7,8,9,10,11,12,13,14,15,16,17,18,19,20,21}
        \draw [xshift=0cm](\x cm,0pt) -- (\x cm,-3pt)
         node[anchor=north] {$\x$};
          \foreach \y in {1,2,3,4,5,6,7,8,9,10,11,12}
        \draw (1pt,\y cm) -- (-3pt,\y cm) node[anchor=east] {$\y$};
    \node[below=0.2cm, right=4.5cm] at (x axis mid) {$r$};
   \node[left=0.2cm, below=-2.9cm] at (y axis mid) {$a$};%, rotate=90
 %\draw[<-, blue](0.1,0.1)-- node[below=2cm,left=2cm]{$g$} (11,11);   
 %\draw[<-, red](21.5,0.5)-- node[below=2cm,right=2.1cm]{$k$} (11,11);
  \end{tikzpicture} 
\caption{Order 2 - The lattice $T$ admits an embedding in $L$ with orthogonal $S$ such that $\ell(A_S)=\ell(A_T)+1$.} \label{app:triangle1}
\end{figure}
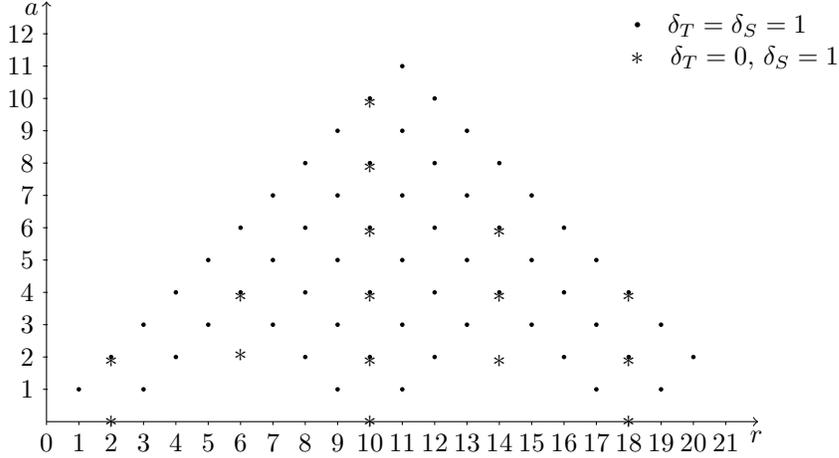

 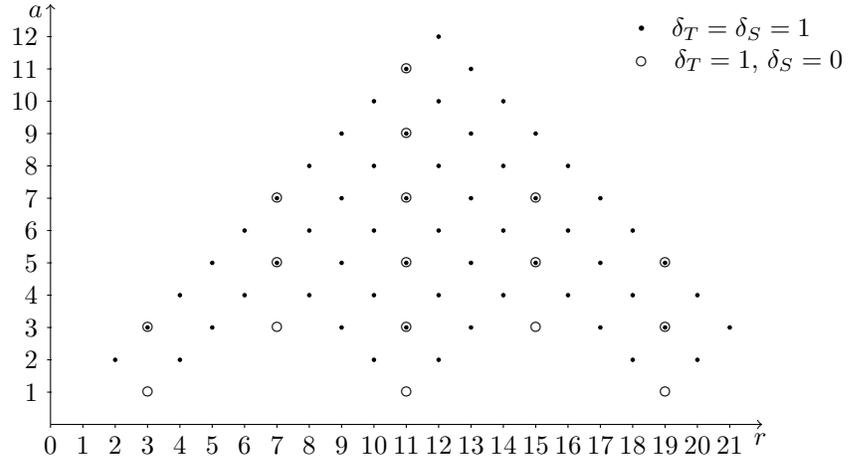
\begin{figure}[!ht]
\hfill $\begin{array}{l} \sbt \quad \delta_T=\delta_S=1\\ 
\circ \quad \delta_T=1,\, \delta_S=0\end{array}$ 
\vspace{-1.1cm}

\begin{tikzpicture}[scale=.43]
\filldraw [black] 
%(1,1) circle (1.5pt)  %node[below=-0.55cm]{10}
%(2,0) node[below=-0.20cm]{*} 
(2,2) circle (1.5pt) %node[below=-0.15cm]{*}%node[below=-0.5cm]{9} 
 (3,1) node{$\circ$} 
 (3,3) circle (1.5pt)node{$\circ$} %node[below=-0.5cm]{8}
 (4,2) circle (1.5pt)
(4,4) circle (1.5pt)%node[below=-0.5cm]{7}
(5,3) circle (1.5pt)
(5,5) circle (1.5pt)
(6,4) circle (1.5pt)%node[below=-0.15cm]{*}
% (6,2)    node[below=-0.23cm]{*}
(6,6) circle (1.5pt)%node[below=-0.5cm]{5}
(7,3) node{$\circ$} 
(7,5) circle (1.5pt)node{$\circ$} 
(7,7) circle (1.5pt)node{$\circ$} %node[below=-0.5cm]{4}
% (8,2) circle (1.5pt)
(8,4) circle (1.5pt)
(8,6) circle (1.5pt)
(8,8) circle (1.5pt)%node[below=-0.5cm]{3}
% (9,1) circle (1.5pt)
(9,3) circle (1.5pt)
(9,5) circle (1.5pt)
(9,7) circle (1.5pt)
(9,9) circle (1.5pt)%node[below=-0.5cm]{2}
% (10,0)  node[below=-0.20cm]{*}
(10,2) circle (1.5pt)%node[below=-0.15cm]{*}
(10,4) circle (1.5pt)%node[below=-0.15cm]{*}
(10,6) circle (1.5pt)%node[below=-0.15cm]{*}
(10,8) circle (1.5pt)%node[below=-0.15cm]{*}
(10,10) circle (1.5pt)%node[below=-0.15cm]{*}
(11,1) node{$\circ$} 
(11,3) circle (1.5pt)node{$\circ$} 
(11,5) circle (1.5pt)node{$\circ$} 
(11,7) circle (1.5pt)node{$\circ$} 
(11,9) circle (1.5pt)node{$\circ$} 
(11,11) circle (1.5pt)node{$\circ$}  %node[below=-0.5cm]{0}
(12,2) circle (1.5pt)
(12,4) circle (1.5pt)
(12,6) circle (1.5pt)
(12,8) circle (1.5pt)
(12,10) circle (1.5pt)
(12,12) circle (1.5pt) %node[below=-0.5cm]{1}
(13,3) circle (1.5pt)
(13,5) circle (1.5pt)
(13,7) circle (1.5pt)
(13,9) circle (1.5pt) %node[below=-0.5cm]{2}
(13,11) circle (1.5pt)
% (14,2)  node[below=-0.15cm]{*}
(14,4) circle (1.5pt)%node[below=-0.15cm]{*}
(14,6) circle (1.5pt)%node[below=-0.15cm]{*}
(14,8) circle (1.5pt) %node[below=-0.5cm]{3}
(14,10) circle (1.5pt)
(15,3)node{$\circ$} 
(15,5) circle (1.5pt)node{$\circ$} 
(15,7) circle (1.5pt)node{$\circ$} % node[below=-0.5cm]{4}
(15,9) circle (1.5pt)% node{$\circ$} 
% (16,2) circle (1.5pt)
(16,4) circle (1.5pt)
(16,6) circle (1.5pt)
(16,8) circle (1.5pt) %node[below=-0.5cm]{5}
% (17,1) circle (1.5pt)
(17,3) circle (1.5pt)
(17,5) circle (1.5pt)%node[below=-0.5cm]{6}
%(18,0)  node[below=-0.2cm]{*}
(17,7) circle (1.5pt)
(18,2) circle (1.5pt)%node[below=-0.15cm]{*}
(18,4) circle (1.5pt) %node[below=-0.15cm]{*}
(18,6) circle (1.5pt)
(19,1) node{$\circ$} 
(19,3) circle (1.5pt)node{$\circ$}  %node[below=-0.5cm]{8}
(19,5) circle (1.5pt)node{$\circ$} 
(20,2) circle (1.5pt)
(20,4) circle (1.5pt)
(21,3) circle (1.5pt)%node[below=-0.5cm]{9}
 ; 
\draw[->] (0,0) -- coordinate (x axis mid) (22,0);
    \draw[->] (0,0) -- coordinate (y axis mid)(0,13);
    \foreach \x in {0,1,2,3,4,5,6,7,8,9,10,11,12,13,14,15,16,17,18,19,20,21}
        \draw [xshift=0cm](\x cm,0pt) -- (\x cm,-3pt)
         node[anchor=north] {$\x$};
          \foreach \y in {1,2,3,4,5,6,7,8,9,10,11,12}
        \draw (1pt,\y cm) -- (-3pt,\y cm) node[anchor=east] {$\y$};
    \node[below=0.2cm, right=4.5cm] at (x axis mid) {$r$};
    \node[left=0.2cm, below=-2.9cm] at (y axis mid) {$a$};%, rotate=90
 %\draw[<-, blue](0.1,0.1)-- node[below=2cm,left=2cm]{$g$} (11,11);   
 %\draw[<-, red](21.5,0.5)-- node[below=2cm,right=2.1cm]{$k$} (11,11);
  \end{tikzpicture} 
\caption{Order 2 - The lattice $T$ admits an embedding in $L$ with orthogonal $S$ such that $l(A_S)=l(A_T)-1$ (all realized by natural automorphisms) }\label{app:triangle2}
\end{figure}

\clearpage

\bibliographystyle{amsplain}
\bibliography{BiblioSmithApplications}
\end{document}